\documentclass[a4paper,12pt,reqno]{amsart}

\usepackage[symbol]{footmisc}
\usepackage{url}
\usepackage[utf8]{inputenc}
\usepackage{amsmath,amsfonts,amsthm,amssymb,amscd,enumerate,mathtools}
\usepackage{array}
\usepackage{graphicx}
\usepackage{tikz-cd}
\usepackage{adjustbox}
\usepackage[OT2,T1]{fontenc}
\usepackage[margin=1.3in]{geometry}
\numberwithin{equation}{section} \numberwithin{figure}{section}

\DeclareMathOperator{\Prym}{Prym} 
\DeclareMathOperator{\Tr}{Tr}
\DeclareMathOperator{\A}{A}

\DeclareMathOperator{\Gen}{Gen}
\DeclareMathOperator{\Sing}{Sing}
\DeclareMathOperator{\Ext}{Ext}
\DeclareMathOperator{\dd}{d}
\DeclareMathOperator{\N}{N}
\DeclareMathOperator{\Proj}{Proj} 
\DeclareMathOperator{\Sym}{Sym} 
\DeclareMathOperator{\Ker}{Ker} 
\DeclareMathOperator{\Alb}{Alb} 
\DeclareMathOperator{\Corr}{Corr} 
\DeclareMathOperator{\Ab}{Ab}
\DeclareMathOperator{\Ch}{Ch} 
\DeclareMathOperator{\mult}{mult} 
\DeclareMathOperator{\Pic}{Pic} \DeclareMathOperator{\Div}{Div}
 
\DeclareMathOperator{\Gal}{Gal} 
\DeclareMathOperator{\Aut}{Aut} 
\DeclareMathOperator{\Spec}{Spec}
 \DeclareMathOperator{\rank}{rank}
\DeclareMathOperator{\Hom}{Hom} 
\DeclareMathOperator{\im}{Im}

\DeclareMathOperator{\ord}{ord}

\DeclareMathOperator{\Bl}{Bl} \DeclareMathOperator{\res}{res}
\DeclareMathOperator{\Nm}{Nm} 
 
\DeclareMathOperator{\HH}{H}

\newcommand{\rom}[1]
    {\MakeUppercase{\romannumeral #1}}

\theoremstyle{definition}
\newtheorem{definition}{Definition}[section]
\newtheorem{remark}[definition]{Remark}
\newtheorem{example}[definition]{Example}
\theoremstyle{plain}
\newtheorem{proposition}[definition]{Proposition}
\newtheorem{corollary}[definition]{Corollary}

\newtheorem{lemma}[definition]{Lemma}

\newtheorem{innercustomthm}{Theorem}

\newenvironment{customthm}[1]
  {\renewcommand\theinnercustomthm{#1}\innercustomthm}
  {\endinnercustomthm}

\theoremstyle{definition}
  \newtheorem{innercustomproof}{Proof}
\newenvironment{customproof}[1]
  {\renewcommand\theinnercustomproof{#1}\innercustomproof }
  {\hfill$\square$}

\title{Prym varieties and cubic threefolds over $\mathbb{Z}$}
\author{Tudor Ciurca}
  \address{
	Department of Mathematical Sciences\\
	University of Bath\\
	Claverton Down\\
	Bath\\
	BA2 7AY\\
	UK}
\email{tc703@bath.ac.uk}

\begin{document}

\begin{abstract}
We develop a theory of Prym varieties and cubic threefolds over fields of characteristic $2$. As an application, we prove that smooth cubic threefolds are non-rational over an arbitrary field and solve a conjecture of Deligne regarding arithmetic Torelli maps. We also prove the Torelli theorem for cubic threefolds over arbitrary fields.
\end{abstract}

\maketitle
\thispagestyle{empty}
\tableofcontents

\section{Introduction}

Smooth cubic threefolds were shown to be non-rational in \cite{clem} over the complex field, and later in \cite{murre} and \cite{beau} over fields of characteristic not $2$, using the theory of Prym varieties. We complete this program by extending the result to fields of characteristic $2$.

\begin{customthm}{A}
A smooth cubic threefold over an arbitrary field is not rational.
\end{customthm}

\noindent To achieve this, we develop a theory of Prym varieties in characteristic $2$. In order to apply this theory to a cubic threefold $X$, we need to transform $X$ into a conic bundle over $\mathbb{P}^2$. This is done by projecting away from a line on $X$. We say that such a line $l \subset X$ is \emph{good} if the resulting conic bundle has a smooth discriminant curve, whose fibers are all given by $2$ distinct lines. In the case that $X$ contains a good line, then a Prym variety can be constructed and the proof in \cite{beau} can be adapted to characteristic $2$ with some extra work.

However it turns out that for the Fermat cubic threefold in characteristic $2$, there are no good lines, so we cannot construct a Prym variety. We address this case separately by considering its automorphism group as in \cite{Beauville}. It turns out, over algebraically closed fields, that the Fermat cubic threefold is the only cubic threefold with no good lines (Proposition \ref{nogoodline}). In fact, the cubic threefolds $X$ over an arbitrary field $k$ which have no good lines turn out to be the Hermitian cubic threefolds (Proposition \ref{2bic}). A cubic threefold is Hermitian if it is defined by an equation whose monomials are all non-squarefree, over a field of characteristic $2$. The theory of Hermitian hypersurfaces (also called $q$-bic hypersurfaces) has been studied more broadly in \cite{qbic, cheng}.

The Prym varieties we construct are principally polarised abelian fivefolds, which are intermediate Jacobians of the cubic threefolds. The intermediate Jacobian of a cubic threefold behaves like the Jacobian of a curve, and contains information regarding the rationality of the threefold. We are able to construct intermediate Jacobians of families of cubic threefolds given some constraints on the base. This solves a conjecture of Deligne posed in \cite[3.3]{deligne}, which is as follows. Let $U \subset \mathbb{P}^{34}_{\mathbb{Z}}$ be the locus of smooth cubic forms in $5$ variables. There is a universal family $\mathcal{X} \to U$ of smooth cubic threefolds with the following property. For any scheme $T$ and any closed subscheme $X \subset \mathbb{P}^4_T$ which is a family of smooth cubic threefolds over $T$, there is a (non-unique) morphism $f:T \to U$ so that $X \cong f^*\mathcal{X}$. Then the following result holds.

\begin{customthm}{B}
There exists a principally polarised abelian scheme $(A,\Theta)$ over $U$ which extends the principally polarised abelian scheme $J_{\mathbb{Q}}$ over $U_{\mathbb{Q}}$ having the defining property that for all points $s \in U_{\mathbb{Q}}$, the fibre $(J_{\mathbb{Q}})_s$ is the intermediate Jacobian of the cubic threefold $\mathcal{X}_s$. 
\end{customthm}

\noindent Previously, the best known result in this direction was due to Achter \cite[Theorem B]{achter}, who constructs $(A,\Theta)_{\mathbb{Z}[\frac{1}{2}]}$. We stress that the theory of Prym varieties in characteristic $2$ is a crucial ingredient to this result, since it increases the codimension in $U$ of the locus of cubic threefolds which have no known intermediate Jacobian to at least $2$. We are then in a position to exploit results about extending abelian schemes over a closed subset of codimension at least $2$, in particular \cite[Corollary 6.8]{Faltings} and \cite[Theorem 1.3]{Vasiu}.

Finally we proceed to prove a Torelli theorem for smooth cubic threefolds over arbitrary fields. This was achieved for algebraically closed fields of characteristic not $2$ in \cite{beau}. In \cite[Theorem 1.3]{loughran} the Torelli theorem for arbitrary fields of characteristic $0$ is claimed, however the proof is incomplete. The result \cite[Proposition 3.2]{loughran} shows that the arithmetic Torelli map is universally injective, but a morphism of stacks being universally injective only means that it is injective on geometric points, which is weaker than what the authors require. We remedy this with the following result.

\begin{customthm}{C} 
Let $X,Y$ be smooth cubic threefolds over a field $k$ such that their intermediate Jacobians $JX$ and $JY$ are isomorphic. Then $X \cong Y$.
\end{customthm}

\noindent We do this by proving, for cubic threefolds $X$ over arbitrary fields, that the intermediate Jacobian $J(X)$ is isomorphic to the Albanese of the Fano surface of lines $\Alb(F(X))$. We use this to extend the proof of \cite[Proposition 6]{beau} stating that the projectivisation of the tangent cone to the theta divisor of $J(X)$ at its unique singularity is the cubic threefold $X$, to aritrary base fields. This allows one to recover a cubic threefold from its intermediate Jacobian.

 For both of these results one has to consider the case of the Fermat cubic threefold separately from other cubic threefolds. The intermediate Jacobian of the Fermat cubic is essentially constructed by lifting to characteristic $0$, and \cite[Theorem 5.3]{achter2022functorial} ensures that this is the correct interpretation.

\subsection{Key differences in characteristic $2$}
The overall strategy for non-Hermitian cubic threefolds is similar across all characteristics, but there are some important new features in characteristic $2$. When $\textrm{char}(k)\neq 2$, there is an explicit correspondence between \'etale double covers of a curve $C$ and $2$-torsion points of its Jacobian $JC$ over the algebraic closure $\overline{k}$, as described in \cite[Exercise \rom{4}.2.7]{hartshorne}. This is no longer the case when $\textrm{char}(k) = 2$. This can be attributed to the fact that an \'etale double cover $\pi: \widetilde{C} \to C$ attached to a smooth cubic threefold equiped with a good line is an Artin-Schreier extension. This means that we have a non-split short exact sequence
$$0 \to \mathcal{O}_C \to \pi_* \mathcal{O}_{\widetilde{C}} \to  \mathcal{O}_C \to 0.$$
This is contrary to the case $\textrm{char}(k)\neq 2$, where we obtain a split short exact sequence
$$0 \to \mathcal{O}_C \to \pi_* \mathcal{O}_{\widetilde{C}} \to \eta \to 0$$
with $\eta \in JC[2]$ the torsion point corresponding to $\pi$. 

Despite these differences, if we let $H$ be the hyperplane divisor on $C$, we can still compute $h^0(\pi^*H) = 4$, and that $[\pi^*H]$ is the sole singularity of the theta divisor on the Prym variety, as is done in \cite[Proposition 2]{beau}, although the methods used are different.

 Consider the two connected components $P = P_0 \cup P_1$ of $P = \Ker(\pi_*)+ [\pi^*H]$, the translate of the kernel of $\pi_*: J\widetilde{C} \to JC$ by a theta characteristic. In characteristic $2$ we encounter another issue when proving the following description 
$$P_0(k) = \{[D] \in J_{10}\widetilde{C}(k) : \pi_*D \cong K_C \textrm{ and } h^0(D) \textrm{ even} \},$$
$$P_1(k) = \{[D] \in J_{10}\widetilde{C}(k) : \pi_*D \cong K_C \textrm{ and } h^0(D) \textrm{ odd} \}.$$
In \cite{m-tc}, Mumford proves this assuming that the characteristic of the base field is not $2$. We rework his proof to suit fields of all characteristics. The main obstacle is that there is no longer an equivalence between quadratic forms and symmetric bilinear forms. However we remedy this by realizing that the correct object we need is a quadratic form whose associated quadric hypersurface is smooth. With this slight modification, Mumford's proof follows through. 

For Hermitian cubic threefolds the entire theory breaks down as we can no longer make use of Prym varieties. For the Klein cubic threefold in characteristic $0$, there is an alternative proof of non-rationality by Beauville which makes use of its large automorphism group acting on its intermediate Jacobian. We are able to adapt this proof to all characteristics using $\ell$-adic cohomology. It turns out that this is sufficient, because all Hermitian cubic threefolds are geometrically isomorphic to the Klein cubic.

Nevertheless, it is difficult to understand what the intermediate Jacobian of a Hermitian cubic threefold looks like. To construct it, we lift to characteristic $0$ and use results from \cite{achter2020decomposition}. In order to prove the Torelli theorem in the case $\textrm{char}(k) \neq 2$, a key input is the fact that the Albanese morphism attached to the Fano variety of lines on the cubic threefold is injective. For the case $\textrm{char}(k) = 2$, we introduce a new proof of this fact that relies on work of Raymond Cheng. In \cite[Theorem 6.14]{cheng} he shows that the intermediate Jacobian is purely inseparably isogenous to the product of $5$ elliptic curves, and it is this result which we exploit in our proof.

\subsection*{Outline of the paper}
In \S 2 we construct Prym varieties in all characteristics and define their natural principal polarisation. In \S 3 we describe the Prym variety as a moduli space of line bundles. In \S 4 we study cubic threefolds and show that Hermitian cubic threefolds are exactly the cubic threefolds with no good lines. We also compute the singularities of the theta divisor on the Prym variety. In \S 5 we consider the problem of constructing intermediate Jacobians for families of cubic threefolds over arbitrary bases, and whether they satisfy base change. Over abritrary fields, we show that intermediate Jacobians of cubic threefolds exist and are isomorphic to the Albanese of their Fano surface of lines.  In \S 6 we study applications of our results to Deligne's question on arithmetic Torelli maps, and prove the Torelli theorem for cubic threefolds over an arbitrary field. We also prove that cubic threefolds over arbitrary fields are non-rational.

\subsection{Conventions}
Except for section $4.1$, a curve will always be smooth and projective. Cubic threefolds will always be smooth.

\subsection*{Acknowledgements}
I would like to thank my supervisor Daniel Loughran for many useful comments and the guidance he has provided me over the last three years. I would like to thank Klaus Hulek for suggesting the problem. I would like to thank Will Sawin for outlining the proof of Lemma \ref{upper}, and discussions regarding Artin-Schreier theory. I would also like to thank Raymond Cheng and Martin Gebhard for many useful discussions regarding the Fermat cubic threefold. I would like to thank Olivier Wittenberg for suggesting the construction of the intermediate Jacobian of the Fermat cubic threefold. The author is supported by EPSRC studentship EB-MA1320.

\section{Prym varieties in all characteristics}
In this section we give a construction of Prym varieties in all characteristics.

Whilst preparing this paper, we learnt of the preprint \cite{pprym} which also considers Prym varieties in characteristic $2$. The results obtained there are more general, but for our purposes the results in this chapter suffice. Moreover, the methods used here differ from the methods in \cite{pprym}.

\subsection{Smoothness of the norm map on Jacobians}
In this subsection we are going to consider how a morphism of curves induces a norm map between the respective Jacobians. We will show that in the case of an \'etale cover of curves, the norm map is smooth. As a result we get that the Prym variety exists as an abelian variety.

Let $\pi:\widetilde{C} \to C$ be a finite \'etale morphism of curves of degree $d$ over a field $k$. We have a pullback map $\pi^* : JC \to J\widetilde{C}$ as well as a norm map $\pi_* : J\widetilde{C} \to JC$ between Jacobians, and we have the relation $\pi_* \circ \pi^* = [d]_{JC}$. 

\begin{definition}
The \emph{Prym variety} $\Prym(\widetilde{C}/C)$ is defined as the connected component of the kernel of $\pi_*: J\widetilde{C} \to JC$. Sometimes we will use the notation $P^+ := \Ker(\pi_*)^0$.
\end{definition}

We will show that this is an abelian variety. For this the only technical difficulty is showing that $P^+$ is smooth. We will demonstrate this by showing that $\pi_*$ is a smooth morphism..

\begin{proposition}\label{tspace}
Let $C$ be a curve over a field $k$, and let $JC$ be its Jacobian. Let $C_{\epsilon} := C \times_{k} \Spec(k[\epsilon]/(\epsilon^2))$ be the base change to the dual numbers. Then we have natural isomorphisms
$$\HH^1(C_{\epsilon}, \mathcal{O}^{\times}_{C_{\epsilon}}) \cong \HH^1(C, (\mathcal{O}_C \oplus \epsilon\mathcal{O}_C)^{\times}) \cong  \HH^1(C, \mathcal{O}_C^{\times}) \oplus \epsilon \HH^1(C, \mathcal{O}_C)$$
and
$$ \HH^1(C, \mathcal{O}_C) \cong \Ker(\HH^1(C_{\epsilon}, \mathcal{O}^{\times}_{C_{\epsilon}})  \to  \HH^1(C, \mathcal{O}_C^{\times}) ) \cong T_0JC $$
where the map $\HH^1(C,\mathcal{O}_C) \to \Ker(\HH^1(C_{\epsilon}, \mathcal{O}^{\times}_{C_{\epsilon}}) \to \HH^1(C, \mathcal{O}_C^{\times}) )$ sends $a\mapsto 1 + a\epsilon$.
\end{proposition}
\begin{proof}
This follows from \cite[Proposition 2.1]{MilneJV}.
\end{proof}

\begin{definition}[Trace map on cohomology] \label{cohomologyTrace}
Let $\pi:\widetilde{C} \to C$ be a finite etale cover of curves. Let $k(C)$ be the function field of $C$ and $\underline{k(C)}$ the corresponding constant sheaf. Let $S(C) := \underline{k(C)}/\mathcal{O}_C$. Define $k(\widetilde{C}),\underline{k(\widetilde{C})}, S(\widetilde{C})$ similarly. Consider the following commutative diagram of short exact sequences
    $$
\begin{tikzcd}
0 \arrow{r} &  \pi_*\mathcal{O}_{\widetilde{C}} \arrow{d}{\Tr} \arrow{r} & \pi_*\underline{k(\widetilde{C})} \arrow{d}{\Tr} \arrow{r}  &  \pi_*S(\widetilde{C}) \arrow{d}{\overline{\Tr}}  \arrow{r} & 0 \\
0 \arrow{r} & \mathcal{O}_C  \arrow{r} & \underline{k(C)}  \arrow{r}  &  S(C)  \arrow{r} & 0
\end{tikzcd}
$$
where the first two vertical maps are just the trace maps, inducing the third map $\overline{\Tr}$. The \emph{trace map on cohomology} $$\Tr: \HH^1(\widetilde{C},\mathcal{O}_{\widetilde{C}}) \to \HH^1(C,\mathcal{O}_{C})$$ is then given from the induced long exact sequence.
\end{definition}

\begin{definition}[Group of repartitions]
Let $C$ be a curve. We define the \emph{group of repartitions} as follows
\begin{align*}
R(C) & := \{(f_p)_{p \in C} \in \oplus_{p \in C}k(C) : f_p \in \mathcal{O}_{C,p} \textrm{ for all but finitely many } p \}, \\
R(\mathcal{O}_C) & := \{(f_p)_{p \in C} \in \oplus_{p \in C}k(C) : f_p \in \mathcal{O}_{C,p} \textrm{ for all  } p \}.
\end{align*}
\end{definition}

\begin{proposition}\label{serre}
Let $C$ be a curve over an algebraically closed field $k$. From Definition \ref{cohomologyTrace}, since $\underline{k(C)}$ is flasque, we have $$\HH^1(C,\mathcal{O}_C) \cong \Gamma(S(C))/\im(\underline{k(C)} \to S(C)).$$ Then we have $$\HH^1(C,\mathcal{O}_C) \cong R(C)/(R(\mathcal{O}_C) + k(C))$$ and this isomorphism is induced by the map $\Gamma(S(C)) \to R(C)/R(\mathcal{O}_C)$ sending a section $s$ to $\{s_p\}_{p \in C}$ where $s_p \in S(C)_p \cong k(C)/\mathcal{O}_{C,p}$ at each point $p$. 
\end{proposition}
\begin{proof}
 This is \cite[\rom{2}, Proposition $3$]{serre}.
\end{proof}

\begin{lemma}\label{repartitionsTrace}
Let $\pi:\widetilde{C} \to C$ be a Galois cover of curves with Galois group $G$ over an algebraically closed field $k$. For any $f\in k(\widetilde{C})$ and $q \in \widetilde{C}$, we have $(f_q + \mathcal{O}_{\widetilde{C},q}) \cap k(C) \neq \emptyset$.
\end{lemma}
\begin{proof}
Since $\pi$ is Galois and $k$ is algebraically closed, then fixing any $p \in C$,
\begin{enumerate}
\item For all $q \in \pi^{-1}(p)$ the local rings $\mathcal{O}_{\widetilde{C},q}$ are isomorphic via the Galois action.
\item Any uniformizer for $\mathcal{O}_{C,p}$ becomes a uniformizer for each $\mathcal{O}_{\widetilde{C},q}$.
\end{enumerate}
It suffices to show that $k(C) + \mathcal{O}_{\widetilde{C},q} = k(\widetilde{C})$. Since
$$k(\widetilde{C}) / \mathcal{O}_{\widetilde{C},q} = \left \{ \frac{\sum_{i=0}^{n-1}b_it^i}{t^n}: b_i \in  \mathcal{O}_{C,\N(q)}, n \in \mathbb{Z}_{\geq 1}\right \}$$
it remains to show that $\frac{b}{t^i} \in k(C) + \mathcal{O}_{\widetilde{C},q}$ for all $i$ and $b \in \mathcal{O}_{C,\N(q)}$. However
$$\frac{b}{t^i} = \Tr \left (\frac{b}{t^i} \right ) - \sum_{g \in G: g\neq 1} g \left (\frac{b}{t^i} \right )$$
and $g(\frac{b}{t^i}) \in \mathcal{O}_{\widetilde{C},q}$ for all $g \in G$ so this is indeed true.
\end{proof}

\begin{definition}[Trace map on repartitions]
Let $\pi:\widetilde{C} \to C$ be a Galois cover of curves with Galois group $G$ over an algebraically closed field $k$. We define the \emph{trace map on repartitions} $\Tr_R: R(\widetilde{C})/R(\mathcal{O}_{\widetilde{C}}) \to R(C)/R(\mathcal{O}_C)$ as
$$\Tr_R((f_q)_{q \in \widetilde{C}}) = \left (\sum_{q \in \pi^{-1}(p)} f_q' \right )_{p \in C} + R(\mathcal{O}_C)$$
where arbitrary $f_q' \in (f_q + \mathcal{O}_{\widetilde{C},q}) \cap k(C)$ are chosen for all $q \in \widetilde{C}$. This is well-defined by Lemma \ref{repartitionsTrace}.
\end{definition}

\begin{lemma}\label{traceDualSerreDuality}
Let $\pi:\widetilde{C} \to C$ be a Galois cover of curves with Galois group $G$ over an algebraically closed field $k$. Then the trace map on cohomology
 $$\Tr: \HH^1(\widetilde{C},\mathcal{O}_{\widetilde{C}}) \to \HH^1(C,\mathcal{O}_{C}) $$
\noindent is dual under Serre duality to the pullback map 
$$\pi^*: \HH^0(C, \Omega_C^1) \to \HH^0(\widetilde{C}, \Omega_{\widetilde{C}}^1).$$
\end{lemma}
\begin{proof}
We will use explicit Serre duality described in the language of repartitions \cite[\rom{2}, Proposition $3$]{serre}. The trace map on repartitions 
$$\Tr_R:R(\widetilde{C})/R(\mathcal{O}_{\widetilde{C}})  \to R(C)/R(\mathcal{O}_C) $$
sends $k(\widetilde{C}) + R(\mathcal{O}_{\widetilde{C}})$ to $k(C) + R(\mathcal{O}_{C})$ and so by Proposition \ref{serre} it induces a map 
$$\Tr':  \HH^1(\widetilde{C},\mathcal{O}_{\widetilde{C}}) \to \HH^1(C,\mathcal{O}_{C}).$$
To show that $\Tr = \Tr'$, it suffices to show that the diagram
   $$
\begin{tikzcd}
 \Gamma(\pi_*(S(\widetilde{C}))) \arrow{d} \arrow{r}{\Tr} &\Gamma(S(C)) \arrow{d} \\
R(\widetilde{C})/R(\mathcal{O}_{\widetilde{C}})  \arrow{r}{\Tr} & R(C)/R(\mathcal{O}_C)  
\end{tikzcd}
$$
commutes, which is easy to verify.

 Now we consider explicit Serre duality for curves as in \cite[\rom{2}, Theorem $2$]{serre}. There is a pairing
$$\langle, \rangle : R(C) \times \HH^0(C,\Omega_C^1) \to k$$
sending $((f_p)_{p \in C}, \omega)$ to $\sum_{p \in C} \res_p(f_p\omega)$. This induces a perfect pairing between $\HH^1(C,\mathcal{O}_{C})$ and $\HH^0(C, \Omega_C^1)$, and therefore $\HH^1(C,\mathcal{O}_{C}) \cong \HH^0(C, \Omega_C^1)^{\vee}$. Now
\begin{align*}
\langle \Tr((f_q)_{q \in \widetilde{C}}), \omega \rangle & =\left  \langle \left  (\sum_{q \in \pi^{-1}(p)}f_q' \right )_{p \in C}, \omega  \right \rangle = \\
 &=  \sum_{p \in C}\res_p\left ( \left (\sum_{q \in \pi^{-1}(p)}f_q' \right )\omega \right ) =\\
 & =  \sum_{p \in C}\res_{q}\left ( \left (\sum_{q \in \pi^{-1}(p)}f_q' \right )\pi^*\omega \right )=  \\
& = \sum_{p \in C}\sum_{q \in \pi^{-1}(p)}\res_q(f_q\pi^*\omega) = \langle (f_q)_{q \in \widetilde{C}}, \pi^*\omega \rangle
\end{align*}
so the trace map corresponds under duality to the pullback map on differential forms $\pi^*: \HH^0(C, \Omega_C^1) \to \HH^0(\widetilde{C}, \Omega_{\widetilde{C}}^1)$.
\end{proof}

\begin{proposition}\label{pushforwardsmooth}
Let $\pi:\widetilde{C} \to C$ be a finite \'etale cover of curves. Then $\pi_*: J\widetilde{C} \to JC$ is a smooth morphism.
\end{proposition}
\begin{proof}
We begin by reducing to the case where $\pi$ is a Galois cover of curves. By \cite[0BY1]{sp}, the category of finitely generated field extensions $K/k$ of transcendence degree $1$ is equivalent to the category of regular projective curves and nonconstant morphisms. Therefore $\pi$ corresponds to a separable field extension $\pi': k(C) \to k(\widetilde{C})$, and we can take the Galois closure $g': k(\widetilde{C}) \to L$ of this field extension. There exists a corresponding Galois cover of curves $g:D \to \widetilde{C}$ such that $L \cong k(D)$. Moreover, $\pi \circ g : D \to C$ is also a Galois covering of curves. Suppose that the result is true for Galois covers of curves and consider the following commutative diagram
$$\begin{tikzcd}
 J(D) \arrow{r}{g_*} \arrow{rd}{(\pi\circ g)_*}  & J(\widetilde{C}) \arrow{d}{\pi_*} \\
  & J(C).
\end{tikzcd}
$$
Then we know that $g_*$ and $(\pi \circ g)_*$ are smooth surjective morphisms. Then due to \cite[02K5]{sp}, it follows that $\pi_*$ is also smooth, as required.

It suffices to show that $\pi \times_k \overline{k}$ is a smooth morphism, so we may assume that the base field $k$ is algebraically closed. Let $d$ be the degree of $\pi$. The map $\pi_*$ induces a map 
$$\mathrm{d}\pi_* : T_0J\widetilde{C} \to T_0JC$$
on tangent spaces at the identity. By \cite[Proposition \rom{3}.10.4]{hartshorne}, it suffices to show that $\mathrm{d}\pi_*$ is surjective. A tangent vector $v \in T_0J\widetilde{C}$ corresponds to a morphism 
$$v: \Spec(k[\epsilon]/(\epsilon^2)) \to J\widetilde{C}$$
with image $0$. This in turn corresponds to a line bundle $L_v$ on $\widetilde{C}_{\epsilon} := \widetilde{C}\times_{k}\Spec(k[\epsilon]/(\epsilon^2))$ so that the natural restriction $L_v|_{\widetilde{C}} \cong \mathcal{O}_{\widetilde{C}}$ is trivial. Let $\pi_{\epsilon}: \widetilde{C}_{\epsilon} \to C_{\epsilon}$ be the base change of $\pi$ by the dual numbers $\Spec(k[\epsilon]/(\epsilon^2))$, which is still finite \'etale.  We may write  $L_v = \mathcal{O}(D^v)$ for some relative Cartier divisor $D^v$. Then $\mathrm{d}\pi_*(v) = \pi_* \circ v$ by the universal property of tangent spaces, and $\mathrm{d}\pi_*(L_v) = \mathcal{O}(\pi_{\epsilon}(D^v))$ by the universal property of Picard varieties. 

Lets us examine the cocycle data of $\mathcal{O}(\pi_{\epsilon}(D^v))$ as it relates to the cocycle data of $\mathcal{O}(D^v)$. Let $\widetilde{C}_{\epsilon} = \cup_{i \in I} U_i$ be a trivializing open cover for $L_v$ and suppose that the cocycles are given by $\{g_{ij}\}_{i,j \in I}$ as an element of $\HH^1(\widetilde{C}_{\epsilon}, \mathcal{O}^{\times})$. We define the norm map
$$\Nm:  \HH^1(\widetilde{C}_{\epsilon}, \mathcal{O}^{\times}) \to \HH^1(C_{\epsilon}, \mathcal{O}^{\times})$$
by sending $\{g_{ij}\}_{i,j \in I}$ to $\{\N(g_{ij})\}_{i,j \in I}$, which is cocycle data for the open cover $\{\pi_{\epsilon}(U_i)\}_{i\in I}$ of $C_{\epsilon}$ defining the line bundle $\mathcal{O}(\pi_{\epsilon}(D^v))$. By considering Proposition \ref{tspace}, we have the following commutative diagram
   $$
\begin{tikzcd}
\HH^1(\widetilde{C}, \mathcal{O}_{\widetilde{C}}) \arrow{r}{\sim} \arrow{d}{t} & T_0J\widetilde{C} \arrow{r} \arrow{d}{\mathrm{d}\pi_*} & \HH^1(\widetilde{C}_{\epsilon}, \mathcal{O}_{\widetilde{C}_{\epsilon}}^{\times}) \arrow{d}{\Nm} \arrow{r}  &  \HH^1(\widetilde{C}, \mathcal{O}_{\widetilde{C}}^{\times}) \arrow{d}{\Nm} \\
\HH^1(C, \mathcal{O}_{C}) \arrow{r}{\sim}&  T_0JC  \arrow{r} & \HH^1(C_{\epsilon}, \mathcal{O}_{C_{\epsilon}}^{\times})  \arrow{r}  &  \HH^1(C, \mathcal{O}_{C}^{\times}).
\end{tikzcd}
$$
We claim that $t$ is the trace map on cohomology as defined in Definition \ref{cohomologyTrace}. An element $a \in \HH^1(\widetilde{C},\mathcal{O}_{\widetilde{C}})$ is sent to $1 + \epsilon a \in \HH^1(\widetilde{C}_{\epsilon},\mathcal{O}_{\widetilde{C}_{\epsilon}}^{\times})$ by the composition of the first two horizontal maps. Consider the computation
$$\Nm(1 + \epsilon a) = \prod_{\sigma \in \Gal(\widetilde{C}_{\epsilon}/C_{\epsilon})} \sigma(1 + \epsilon a) = 1 + \epsilon\sum_{\sigma \in \Gal(\widetilde{C}_{\epsilon}/C_{\epsilon})} \sigma(a) = 1 + \epsilon \Tr(a)$$
which shows that $t$ is indeed induced by the trace map $\Tr: \pi_*\mathcal{O}_{\widetilde{C}} \to \mathcal{O}_{C}$, and therefore $t$ is the trace map on cohomology.

By Lemma \ref{traceDualSerreDuality}, it suffices to show that the map $\pi^*:\HH^0(C, \Omega_C^1) \to \HH^0(\widetilde{C}, \Omega_{\widetilde{C}}^1)$ is injective. However, by \cite[Proposition \rom{4}.2.1]{hartshorne}, this follows from the fact that $\pi$ is finite \'etale. This completes the proof.
\end{proof}

\begin{corollary}
Let $\pi: \widetilde{C} \to C$ be a finite Galois cover of curves. The Prym variety $P^+ := \Ker(\pi_*)^0$ is an abelian variety.
\end{corollary}

\subsection{The principal polarisation on the Prym variety}
In this subsection we will equip the Prym variety $\Prym(\widetilde{C}/C)$ with a canonical principal polarisation which is half of the principal polarisation induced by the theta divisor on $J\widetilde{C}$. 

Let $\pi: \widetilde{C} \to C$ be a finite Galois cover of curves. Let $\theta_{\widetilde{C}}$ and $\theta_C$ be the canonical theta divisors on $J_{g(\widetilde{C})-1}\widetilde{C}$ and $J_{g(C)-1}C$ respectively. These induce principal polarisations $\lambda_{\theta_{\widetilde{C}}}: J\widetilde{C} \xrightarrow{\sim} \widehat{J\widetilde{C}}$ and $\lambda_{\theta_{C}}: JC \xrightarrow{\sim} \widehat{JC}$. 

\begin{lemma}\label{norm-pullback duality}
With the notation above, we have
$$\widehat{\pi_*} = \lambda_{\theta_{\widetilde{C}}} \circ \pi^* \circ \lambda_{\theta_C}^{-1},$$
$$\widehat{\pi^*} = \lambda_{\theta_C} \circ \pi_* \circ \lambda_{\theta_{\widetilde{C}}}^{-1}.$$
\end{lemma}
\begin{proof}
This is shown in \cite[Section 1]{m-pv}. Although \cite{m-pv} assumes that $\textrm{char}(k) \neq 2$ for most of the paper, this particular result is independent of the characteristic of $k$ and only relies on standard facts about abelian varieties and polarisations.
\end{proof}

\begin{proposition}\label{rosen}
Let $\pi : \widetilde{C} \to C$ be a Galois covering of curves with abelian Galois group $G$. Let $\pi_* : J\widetilde{C} \to JC$ be the norm map on Jacobians. Then the number of connected components $|\Ker(\pi_*)/\Ker(\pi_*)^0|$ of the kernel of $\pi_*$ is exactly $|G|$.
\end{proposition}
\begin{proof}
This is the main theorem of \cite{rosen_1983}.
\end{proof}

It follows that $\Ker(\pi_*)$ has two connected components $P^+ \cup P^-$. We want to produce a principal polarisation on $P^+$. First, let us calculate the dimension of $P^+$. Let $g = g(C)$ be the genus of $C$. By Riemann-Hurwitz-Hasse, we have $\widetilde{g} := g(\widetilde{C}) = 2g - 1$. Since $\dim(P^+) + \dim(JC) = \dim(J\widetilde{C})$ this gives us
$$\dim(P^+) = \dim(J\widetilde{C}) - \dim(JC) =  \widetilde{g} - g = g - 1.$$

\begin{proposition}\label{bound}
Consider the setup in Proposition \ref{rosen}, and let $P^+$ be the Prym variety with its natural inclusion $i: P^+ \to J\widetilde{C}$. Let $g = g(C)$ be the genus of the ground curve. Then the map $\pi^* + i : JC \times P^+ \to J\widetilde{C}$ is an isogeny with kernel $H$, and $H$ is a finite group scheme whose degree satisfies
$$|G|^{2g-1} \leq |H| \leq |G|^{2g}.$$
\end{proposition}
\begin{proof}

Let $S$ be a connected $k$-scheme. Let $ (x,y) \in H (S)$ be an $S$-valued point. Then we must have $i(y) = - \pi^*(x)$, so in particular $y$ is uniquely determined by $x$ since $i$ is a monomorphism. As a result, $|G|x = \pi_*\pi^*x = 0$ must hold. This establishes the upper bound, since the first projection $\pi_1 : H \to JC$ factors through $JC[|G|]$, and the degree of the group scheme $JC[|G|]$ is $|JC[|G|]| = |G|^{2g}$. This also proves that $H$ is finite and that $\pi^* + i$ is an isogeny. 

For the lower bound, consider any connected $k$-scheme $S$. The group scheme $JC[|G|]$ has degree $|G|^{2g}$ and $\pi^*(JC[|G|](S)) \subset \Ker(\pi_*)(S)$. By Proposition \ref{rosen}, $\Ker(\pi_*)$ has $|G|$ connected components. Let us label the connected components $K_1 , \dotsc , K_s$ which contain some point of $\pi^*(JC[|G|](S))$. Then $\pi^*(JC[|G|](S)) \cong \bigoplus_{i=1}^sK_i(S) \cap \pi^*(JC[|G|](S))$ and the components $K_i(S) \cap \pi^*(JC[|G|](S))$ are all isomorphic by translation. After changing labels we may assume $K_0 = P^+$. As a result 
$$H(S) = \{p \in JC[|G|](S) : \pi^*p \in P^+(S)\}$$
is bounded below by $|JC[|G|](S)|/s$, and $s$ is at most $|G|$. This establishes the lower bound of $|G|^{2g -1}$.
\end{proof}

\begin{proposition}\label{algequiv}
If $\pi:\widetilde{C} \to C$ is an \'etale double cover, then there exists a divisor $D$ on $P^+$ which induces a principal polarisation $\lambda_{D}:P^+  \xrightarrow{\sim} \widehat{P^+}$. Moreover, $D$ can be chosen such that $\lambda_{\widetilde{\theta}}|_{P^+} = 2\lambda_{D}$.
\end{proposition}
\begin{proof}
We define a polarisation on $JC \times P^+$ as follows
$$\lambda : JC \times P^+ \xrightarrow{\pi^* + i} J\widetilde{C} \xrightarrow{\lambda_{\widetilde{\theta}}} \widehat{J\widetilde{C}} \xrightarrow{(\widehat{\pi^*},\widehat{i})} \widehat{JC} \times \widehat{P^+}$$
which is the pullback of the principal polarisation $\lambda_{\widetilde{\theta}}$ across $\pi^* + i$. Recall that $\pi^* + i$ is an isogeny with kernel $H$. Therefore
$$\deg(\pi^* + i)\chi(\mathcal{O}(\widetilde{\theta})) = \chi(\mathcal{O}((\pi^* + i)^{-1}(\widetilde{\theta}))).$$
By taking squares of both sides and using the Riemann-Roch theorem for line bundles on abelian varieties (see \cite[Page 150]{m-av}), we get
$$|H|^2 = \deg(\pi^* + i)^2 = \deg(\lambda).$$
Now we prove that $\lambda$ splits into a product of polarisations. We have 
$$\lambda = \begin{pmatrix}
a & b \\
c & d
\end{pmatrix} := \begin{pmatrix}
\widehat{\pi^*} \circ \lambda_{\widetilde{\theta}} \circ \pi^* & \widehat{\pi^*} \circ \lambda_{\widetilde{\theta}} \circ i \\
\widehat{i} \circ \lambda_{\widetilde{\theta}} \circ \pi^* & \widehat{i} \circ \lambda_{\widetilde{\theta}} \circ i
\end{pmatrix}$$
and since polarisations are symmetric, we can compute
$$\widehat{\lambda} = \widehat{\begin{pmatrix}
a & b \\
c & d
\end{pmatrix}} = \begin{pmatrix}
\widehat{a} & \widehat{c} \\
\widehat{b} & \widehat{d}
\end{pmatrix} = \begin{pmatrix}
a & b \\
c & d
\end{pmatrix} = \lambda.$$
However, $b = \widehat{\pi^*} \circ \lambda_{\widetilde{\theta}} \circ i = \lambda_{\theta} \circ \pi_* \circ i = 0$ by Lemma \ref{norm-pullback duality}, and $c = \widehat{b} = 0$. We conclude that $\lambda$ is diagonal. As a result, we obtain 
$$\deg(\lambda) = \deg(\widehat{\pi^*} \circ \lambda_{\widetilde{\theta}} \circ \pi^*)\deg(d) = \deg(2\lambda_{\theta})\deg(d) = 2^{2g}\deg(d)$$
where $d = \lambda_{\widetilde{\theta}}|_{P^+}$ is the polarisation on $P^+$ which we are interested in. From Proposition \ref{bound} we obtain
$$2^{4g-2} \leq |H|^2 = 2^{2g}\deg(d) \leq 2^{4g}.$$
We claim that $\Ker(d) \leq P^+[2]$. Suppose $d(s) = \widehat{i} ( \lambda_{\widetilde{\theta}}( i(s))) = 0$. Then $\widehat{\pi^*} ( \lambda_{\widetilde{\theta}}( i(s))) =  \lambda_{\theta}(\pi_*(i(s))) = 0$ also and so $s \in \Ker(\widehat{\pi^* + i} \circ \lambda_{\widetilde{\theta}} \circ i)$. However $i$ and $\lambda_{\widetilde{\theta}}$ are injective homomorphisms whilst $\Ker(\widehat{\pi^* + i}) \leq \widehat{J\widetilde{C}}[2]$, so it follows that $2s = 0$ as required. Thus $\deg(d) \leq 2^{2(g-1)}$ and so $\Ker(d) = P^+[2]$. By the universal property of kernels we can factor $d = 2\lambda_{D}$ where $\lambda_{D}$ is a principal polarisation on $P^+$ for some divisor $D$, as required.
\end{proof}

\section{Quadric hypersurfaces and their generators}
As an abelian subvariety of the Jacobian $J\widetilde{C}$, the Prym variety $\Prym(\widetilde{C}/C)$ should parametrise certain line bundles on $\widetilde{C}$. In this section we will give such a description in all characteristics. This is necessary in order to construct birational invariants of cubic threefolds later in the paper.

\begin{definition}
Let $Q$ be a smooth quadric hypersurface in $\mathbb{P}_k^{2n-1}$. A \emph{generator} of $Q$ is a maximal isotropic subspace of $Q$. This is necessarily an $(n-1)$-plane.
\end{definition}

\begin{proposition}\label{componentsGenerators}
Let $Q$ be a smooth quadric hypersurface in $\mathbb{P}_k^{2n-1}$. Then $\Gen(Q)$, the scheme of generators of $Q$, has two disjoint open components $\Gen(Q)_0$, $\Gen(Q)_1$. Moreover, For generators $g \in \Gen(Q)_a$, $h \in \Gen(Q)_b$ we have
$$\dim(g \cap h) \equiv \dim(g)+a+b \pmod{2}.$$
\end{proposition}
\begin{proof}
From \cite[Proposition 1.12]{deligneSGA}, there is a continuous morphism $e:\Gen(Q) \to \mathbb{Z}/2\mathbb{Z}$ with the property that given generators $g \in \Gen(Q)_a$, $h \in \Gen(Q)_b$ then $e(h) = e(g)$ if and only if $\dim(g/g\cap h)$ is even. This proves the statement.
\end{proof}

\subsection{Quadratic forms on vector bundles and Wirtinger's theorem}
In this section we want to extend the results of \cite[Section 1]{m-tc} to fields of characteristic $2$. Aside from Lemma \ref{char2duality}, all results in this subsection follow through as in \cite{m-tc} for fields of characteristic $2$. Nevertheless we choose to present these results to the reader in the detail which is missing from \cite{m-tc}. 

\begin{definition}
Let $V$ be a vector bundle on a scheme $X$. A \emph{quadratic form} on $V$ is a morphism of sheaves $q:V \to L$ into some line bundle $L$ which can be factored as $$V \xrightarrow{\Delta} V \otimes V \xrightarrow{b} L$$ for some linear map $b$ and the diagonal $\Delta$. We use the following notation for the total space of a vector bundle and its projectivization
$$\mathbb{A}(V) := \underline{\Spec}_X(\Sym(V^{\vee})),$$
$$\mathbb{P}(V) := \underline{\Proj}_X(\Sym(V^{\vee})).$$
The quadratic form $q$ induces a morphism $\mathbb{A}(V) \to \mathbb{A}(L)$ between total spaces over $X$. We say $q$ is \emph{smooth} if $\mathbb{V}_+(q) \subset \mathbb{P}(V)$ is smooth over $X$.

We say $q$ is \emph{symmetric} if $b$ is symmetric for some choice of $b$. We say $q$ is \emph{non-degenerate} if for some choice of $b$, we have that $b_t$ is non-degenerate at all points $t \in X$.
\end{definition}

\begin{remark}
Consider the norm $\N: \mathbb{F}_4 \to \mathbb{F}_2$. Writing $\mathbb{F}_4 = \mathbb{F}_2[t]/(t^2+t+1)$ gives us a natural basis $\{1,t\}$ over $\mathbb{F}_2$ upon which the norm acts via 
$$\N(a,b) := \N(a + bt) = (a+b+bt)(a+bt) = a^2 + ab + b^2.$$
One can verify that we cannot factor $\N$ as $\mathbb{F}_4 \xrightarrow{\Delta} \mathbb{F}_4 \otimes \mathbb{F}_4 \xrightarrow{b} \mathbb{F}_2$ such that $b$ is linear and symmetric. This shows that there are quadratic forms which are not symmetric. 

Consider the trace pairing $\mathrm{T}: \mathbb{F}_4 \xrightarrow{\Delta} \mathbb{F}_4 \otimes \mathbb{F}_4 \xrightarrow{\Tr \circ m} \mathbb{F}_2$. Then
$$\mathrm{T}(a,b) =\Tr(m((a+bt)\otimes(a + bt))) = \Tr(a^2 + b^2 + b^2t) = a^2.$$
This shows that non-degenerate quadratic forms can be non-smooth.
\end{remark}

Let $\pi: \widetilde{C} \to C$ be an \'etale double cover of curves over an algebraically closed field $k$. The natural inclusion $\Ker(\pi_*) \to J\widetilde{C}$ produces, via the universal property of Picard varieties, a family of line bundles $\mathcal{L}$ on $\Ker(\pi_*) \times \widetilde{C}$.  The pushforward $E:=(1,\pi)_*(\mathcal{L})$ is a rank $2$ vector bundle on $\Ker(\pi_*) \times C$.

\begin{lemma}\label{quadform}
We can equip the vector bundle $E$ on $\Ker(\pi_*) \times C$ with a smooth quadratic form $q: E \to \mathcal{O}_{\Ker(\pi_*) \times C}$.
\end{lemma}
\begin{proof}
Since $E$ is a sheaf of \'etale $\mathcal{O}_{\Ker(\pi_*) \times C}$-algebras, there is a norm map $q: E \to \mathcal{O}_{\Ker(\pi_*) \times C}$. On an affine patch $U \cong \Spec(A) \subset \Ker(\pi_*) \times C$, this is given by $q(x) = \N_{E(U)/A}(x)$ for all $x \in E(U)$. It is a quadratic form since we can write it as $E \xrightarrow{\Delta} E \otimes E \xrightarrow{b} \mathcal{O}_{\Ker(\pi_*) \times C}$ where $b$ sends $x \otimes y$ to $x \cdot \sigma(y)$, if $\sigma \in \Gal(k(\widetilde{C})/k(C))$ is the nontrivial element. 

We show that this quadratic form is smooth. Take any closed point $t \in \Ker(\pi_*) \times C$ and consider the quadratic form $q_t: E_t \cong k^2 \to k$ induced by pulling back to $t$. We are to show that $q_t$ is smooth outside $0$. However $E_t$ is the unique \'etale degree $2$ extension of $k$. The norm map $q_t$ is just the multiplication map $m: k \times k \to k$, which is smooth outside $0$. Finally, $q$ is a flat morphism due to miracle flatness, since all the fibres are one-dimensional. It follows that $\mathbb{V}_+(q)$ is flat over $\Ker(\pi_*) \times C$, since $q$ is flat. 
\end{proof}

\begin{definition}
A \emph{theta characteristic} on a curve $C$ is some line bundle $\mathcal{L}_0$ such that $\mathcal{L}_0^{2} \cong \Omega^1_C$. We say $\mathcal{L}_0$ is even (respectively, odd) if $\dim(\Gamma(\mathcal{L}_0))$ is even (respectively, odd).
\end{definition}

Fix some theta characteristic $\mathcal{L}_0$ on $C$, which is possible from the last paragraph of \cite[page 191]{m-tc}. Then $L_0 := \pi_2^*\mathcal{L}_0$ is a line bundle on $\Ker(\pi_*) \times C$, and there exists an isomorphism $\phi: L_0^2 \xrightarrow{\sim} \Omega^1_{\Ker(\pi_*) \times C/\Ker(\pi_*)}$. We can upgrade $q$ to a quadratic form $Q : E' :=  E \otimes L_0 \to \Omega^1_{\Ker(\pi_*) \times C/\Ker(\pi_*)}$ as follows
$$Q: E \otimes L_0\xrightarrow{\Delta} (E \otimes E) \otimes (L_0 \otimes L_0) \xrightarrow{b \otimes \phi}\Omega^1_{\Ker(\pi_*) \times C/\Ker(\pi_*)}.$$
Since tensoring $E$ by $L_0$ does not change the situation locally, it follows that $Q$ is smooth. 

\begin{lemma}\label{isotropicSubspaces}
Choose $N$ disjoint sections $s_1 , \dotsc , s_N : \Ker(\pi_*) \to \Ker(\pi_*)\times C$ of $\pi_1: \Ker(\pi_*)\times C \to \Ker(\pi_*)$, by choosing corresponding points $p_1 , \dotsc , p_N$ on $C$. Let $A$ be the effective relative Cartier divisor corresponding to $\sum_{i=1}^Ns_i(\Ker(\pi_*))$, and let $B = \sum_{i=1}^Np_i$. Consider
$$W_1 = (\pi_1)_*(E'(A)),$$
$$W_2 = (\pi_1)_*(E'/E'(-A)),$$
$$V = (\pi_1)_*(E'(A)/E'(-A)).$$
Then for $N$ sufficiently large, $W_1,W_2,V$ are locally free coherent sheaves, and $W_1,W_2 \subset V$ with $W_1 \cap W_2 = (\pi_1)_*E'$.
\end{lemma}
\begin{proof}
This follows from \cite[Section 1, Pages 182-183]{m-tc}. The arguments used are cohomological in nature and do not depend on the characteristic of the base field.
\end{proof}

\noindent We want to show that $W_1,W_2$ are generators of some quadratic form on $V$.

\begin{lemma}\label{char2duality}
We have $E' \cong \mathcal{H}om(E',\Omega^1_{\Ker(\pi_*)\times C / \Ker(\pi_*)})$ as vector bundles.
\end{lemma}
\begin{proof}
In \cite[end of page 183]{m-tc}, Mumford achieves this by polarising the quadratic form $Q$, which we cannot do in characteristic $2$. Instead, consider the trace pairing 
$$\Tr \circ m: E \otimes E \to \mathcal{O}_{\Ker(\pi_*) \times C}$$ 
which is symmetric and non-degenerate, since $\pi$ is \'etale. We may upgrade this to a symmetric non-degenerate pairing $E' \otimes E' \to \Omega^1_{\Ker(\pi_*)\times C / \Ker(\pi_*)}$ by tensoring by $L_0^2$. By adjunction this produces a morphism of vector bundles
$$\psi: E' \to \mathcal{H}om(E', \Omega^1_{\Ker(\pi_*)\times C / \Ker(\pi_*)}).$$
Let $x \in \Ker(\pi_*)\times C$ be a closed point. Then 
$$\mathcal{H}om(E', \Omega^1_{\Ker(\pi_*)\times C / \Ker(\pi_*)})_x \cong \Hom(E'_x, (\Omega^1_{\Ker(\pi_*)\times C / \Ker(\pi_*)})_x)$$
so it follows by non-degeneracy that $\psi$ is an isomorphism at all closed points. It follows that $\psi$ is an isomorphism because it is a linear map of vector bundles.
\end{proof}

\begin{lemma}
We have $\dim(W_1) = \dim(W_2) = 2N$ and $\dim(V) = 4N$.
\end{lemma}
\begin{proof}
This follows from the arguments at the end of \cite[page 183]{m-tc} and at the beginning of \cite[page 184]{m-tc}.
\end{proof}

We describe a quadratic form $\mathcal{Q} : V \to \mathcal{O}_{\Ker(\pi_*)}$ as follows. There is a smooth quadratic form  $Q(A): E'(A) \to \Omega^1_{\Ker(\pi_*) \times C/\Ker(\pi_*)}(2A)$ given by tensoring $Q$ by $\mathcal{O}_{\Ker(\pi_*) \times C}(A)$. Write
$$Q(A): E'(A) \xrightarrow{\Delta} E'(A) \otimes E'(A) \xrightarrow{b} \Omega^1_{\Ker(\pi_*) \times C/\Ker(\pi_*)}(2A).$$
Let $\eta_i$ be the generic point of the Weil divisor $s_i(\Ker(\pi_*))$. For an open subset $U$ of $\Ker(\pi_*)$, let $s$ be a section of $V(U) = (E'(A)/E'(-A))(\pi_1^{-1}(U))$. Choose representatives $s_{\eta_i} \in E'(A)(\eta_i)$ for the restriction of $s$ modulo $\eta_i$. Then we set 
$$\mathcal{Q}(s) = \sum_{i=1}^N\res_{\eta_i}(Q(A)(s_{\eta_i}))$$
which is a quadratic form since it factors as $V \xrightarrow{\Delta} V \otimes V \xrightarrow{\beta} \mathcal{O}_{\Ker(\pi_*)}$ where
$$\beta (x \otimes y) := \sum_{i=1}^N\res_{\eta_i}(b(x_{\eta_i}\otimes y_{\eta_i})).$$
This is well-defined because the residue maps $\res_{\eta_i}$ do not care about the choice of representatives $x_{\eta_i},y_{\eta_i}$.

\begin{lemma}
The quadratic form $\mathcal{Q}: V \to \mathcal{O}_{\Ker(\pi_*)}$ is smooth.
\end{lemma}
\begin{proof}
We get quadratic forms $Q(A)_x$ by pulling $Q(A)$ back to $\{x\}\times C$, and $(Q(A)_x)_{(p_i)}$ by taking the stalk at the point $p_i$. Then $(Q(A)_x)_{(p_i)} = (Q_x(B))_{(p_i)}$ which we can factor as
$$ (E'_x(B))_{(p_i)} \xrightarrow{\Delta}  (E'_x(B))_{(p_i)} \otimes  (E'_x(B))_{(p_i)} \xrightarrow{b_i} \Omega^1_{C/k}(2B)_{(p_i)}.$$
To prove the lemma it suffices to show that $\mathcal{Q}$ is smooth outside $0$ at every closed point $x \in \Ker(\pi_*)$. Flatness will follow autmatically by miracle flatness, as in the proof of Lemma \ref{quadform}. The quadratic form $\mathcal{Q}_x$ is
$$V_x \xrightarrow{\Delta} V_x \otimes V_x \xrightarrow{\beta_x} k.$$
Fix uniformizers $t_i$ for the curve $C$ at the points $p_i$.  Consider an arbitrary element $a \in E'_x(B)/E'_x(-B)$. For each point $p_i$ of $B$ we can take a representative $a_{(p_i)} \in (E'_x(B))_{(p_i)}$ for $a$ at the stalk at $p_i$, which has some expansion 
$$a_{(p_i)} = \sum_{j\geq -1} (x_{i,j},y_{i,j})t_i^j$$
for elements $(x_{i,j},y_{i,j}) \in E'_x(p_i)/p_iE'_x(p_i) \cong k^2$. Then we may write
$$\beta_x (u \otimes v) := \sum_{i=1}^N\res_{p_i}(b_i(u_{(p_i)}\otimes v_{(p_i)})).$$
For each $a \in V_x$, we obtain

\vspace{5pt}
\begin{tabular}{lrl}
 \( \Delta(a_{(p_i)}) = \) & \( ((x_{i,-1},y_{i,-1})\otimes(x_{i,-1},y_{i,-1}))  t^{-2} \) & \( +\) \\
 & \(+ ((x_{i,-1},y_{i,-1}) \otimes (x_{i,0},y_{i,0}) +  (x_{i,0},y_{i,0}) \otimes (x_{i,-1},y_{i,-1}))  t^{-1}\) & \( + \cdots \)
\end{tabular}

\noindent Write $b_{p_i} = (b_i)_{p_i}$. Since $\res_{p_i}\circ b_i$ only cares about the $t^{-1}$ coefficient, we get
\begin{align*}
 \res_{p_i}(b_i(\Delta(a_{(p_i)}))) = & b_{p_i}((x_{i,-1},y_{i,-1}) \otimes (x_{i,0},y_{i,0}) + (x_{i,0},y_{i,0}) \otimes (x_{i,-1},y_{i,-1})) \\
 = &b_{p_i}(((x_{i,-1},y_{i,-1}) +  (x_{i,0},y_{i,0})) \otimes ((x_{i,-1},y_{i,-1}) +  (x_{i,0},y_{i,0})) - \\
& - ((x_{i,-1},y_{i,-1}) \otimes (x_{i,-1},y_{i,-1}))  -  ((x_{i,0},y_{i,0}) \otimes (x_{i,0},y_{i,0}))  ) \\
= & \N((x_{i,-1} + x_{i,0},y_{i,-1} + y_{i,0} )  ) - \N((x_{i,-1},y_{i,-1})  ) - \N((x_{i,0},y_{i,0} )  )\\
= & x_{i,-1}y_{i,0} + x_{i,0}y_{i,-1}
\end{align*} 
because $(Q_x(B))_{p_i} : k^2 \to k$ is the norm map. Therefore we deduce that the quadratic form $\mathcal{Q}_x: k^{4N} \to k$ looks like
$$\mathcal{Q}_x((a_{i,j})_{1 \leq i \leq 4, 1 \leq j \leq N}) = \sum_{j=1}^Na_{1,j}a_{4,j} + a_{2,j}a_{3,j}$$
which is indeed smooth outside $0$.
\end{proof}

\begin{lemma}\label{int}
The subspaces $W_1, W_2 \subset V$ are generators for the quadratic form $\mathcal{Q}$. Moreover, for any point $x \in \Ker(\pi_*)$, we have $W_{1,x} \cap W_{2,x} = \Gamma(E'_x)$.
\end{lemma}
\begin{proof}
The beginning of \cite[page 184]{m-tc} explains why $W_1,W_2$ are maximal isotropic subspaces for $\mathcal{Q}$. From Lemma \ref{isotropicSubspaces}, we have $W_1 \cap W_2 = (\pi_1)_*E'$ and so 
 $$W_{1,x} \cap W_{2,x} = ((\pi_1)_*E')_x = \Gamma(E'_x). \qedhere$$
\end{proof}

With this at our disposal, we get the following key result.

\begin{proposition}\label{locconst}
The map $\Ker(\pi_*) \to \mathbb{Z}/2\mathbb{Z}$ given by 
$$s \mapsto \dim(\Gamma(E'_s)) \pmod{2}$$
is locally constant.
\end{proposition}
\begin{proof}
By Proposition \ref{componentsGenerators}, for each closed point $x \in \Ker(\pi_*)$, the scheme $\Gen(\mathcal{Q}_x)$ has two connected components which vary continuously in $x$. As a result, if the generators $W_{1,x}$, $W_{2,x}$ lie in the same connected component, then $W_{1,y}$, $W_{2,y}$ lie in the same connected component for all $y$ in the connected component of $x$ in $\Ker(\pi_*)$. The result then follows from Proposition \ref{componentsGenerators} and Lemma \ref{int}.
\end{proof}

\subsection{Applications to Prym varieties}
In this subsection we will say something about the dimensions of spaces of global sections of line bundles parametrized by the Prym variety. This lets us describe the singularities of the theta divisor on the Prym variety due to a theorem of Riemann-Kempf.

Let $\pi: \widetilde{C} \to C$ be an \'etale double cover of curves of genera $\widetilde{g},g$ respectively. It follows that $\widetilde{g} = 2g - 1$.
\begin{proposition}\label{parity}
Let $\sigma: \widetilde{C} \to \widetilde{C}$ be the involution associated to the double cover. Let $x \in \widetilde{C}$ be a closed point. Then
\begin{enumerate}
\item If  $D \in P^+(k)$ then $D + x - \sigma(x) \in P^-(k)$.
\item If $D \in P^-(k)$ then $D + x - \sigma(x) \in P^+(k)$.
\end{enumerate}
\end{proposition}
\begin{proof}
This is \cite[step 3, page 188]{m-tc}, which follows from the previous steps. All the steps are valid when $\textrm{char}(k) = 2$.
\end{proof}

\begin{proposition}[Wirtinger's Theorem]\label{lc}
For any $s \in \Ker(\pi_*)(k)$, let $L_s$ be the line bundle on $\widetilde{C}$ associated to $s$. Let $\mathcal{L}_0$ be a theta characteristic on $C$. Then the function $\Ker(\pi_*)(k) \to \mathbb{Z}/2\mathbb{Z}$ given by $$s \mapsto \dim(\Gamma(\widetilde{C}, L_s\otimes \pi^*\mathcal{L}_0)) \pmod{2}$$
 is locally constant and takes different values on $P^+(k)$ and $P^-(k)$.
\end{proposition}
\begin{proof}
This map is locally constant by Proposition \ref{locconst}, since $\Gamma(E'_s) = \Gamma(\widetilde{C}, L_s\otimes \pi^*\mathcal{L}_0)$ for all $s$. By \cite[step 2, page 187]{m-tc}, combined with Proposition \ref{parity}, this map takes different values on $P^+(k)$ and $P^-(k)$.
\end{proof}

We return to the study of the principal polarisation on the Prym variety. Recall from Proposition \ref{algequiv} that $\lambda_{\widetilde{\theta}}|_{P^+} = 2\lambda_{D}$ for some divisor $D$ on $P^+$. We wish to upgrade this statement to an equality of divisors.

\begin{proposition}[Riemann-Kempf]\label{kempf}
Let $C$ be a curve of genus $g$ with canonical theta divisor $\theta \subset J_{g-1}C$. For any point $\alpha \in J_{g-1}C(k)$, let $L_{\alpha}$ be the corresponding line bundle on $C$. Then
$$\dim(\Gamma(L_{\alpha})) = \mult_{\theta}(\alpha).$$
\end{proposition}
\begin{proof}
This follows from \cite[Corollary, page 184]{gk} whose proof is independent of the characteristic of $k$. In the proof, note that $\mathrm{index}(L_{\alpha}) = h^0(L_{\alpha})$.
\end{proof}

\begin{definition}
Fixing a theta characteristic $\mathcal{L}_0$ on $C$ will also give us a theta characteristic $\pi^*\mathcal{L}_0$ on $\widetilde{C}$. Suppose that $\pi^*\mathcal{L}_0$ is an even theta characteristic. Define $P_0 := P^+ + [\pi^*\mathcal{L}_0] \subset J_{\widetilde{g}-1}\widetilde{C}$ and $P_1 := P^- + [\pi^*\mathcal{L}_0] \subset J_{\widetilde{g}-1}\widetilde{C}$. The $k$-points of these subschemes are given by
$$P_0(k) = \{[D] \in J_{\widetilde{g} -1}\widetilde{C}(k) : \pi_*D \cong K_C \textrm{ and } h^0(D) \textrm{ even} \},$$
$$P_1(k) = \{[D] \in J_{\widetilde{g} - 1}\widetilde{C}(k) : \pi_*D \cong K_C \textrm{ and } h^0(D) \textrm{ odd} \}.$$
\end{definition}

\begin{proposition}\label{theta}
Let $\pi: \widetilde{C} \to C$ be an \'etale double cover of curves. Let $\widetilde{\theta} \subset J_{10}\widetilde{C}$ be the canonical theta divisor. Then there is some effective divisor $D$ such that $\widetilde{\theta}|_{P_0} = 2D$. 
\end{proposition}
\begin{proof}
The divisor $\widetilde{\theta}|_{P_0}$ on $P_0$ is effective, and its $k$-points are given by
$$\widetilde{\theta}|_{P_0}(k) = \{[D] \in J_{\widetilde{g} - 1}\widetilde{C}(k) : \pi_*D \cong K_C \textrm{ and } h^0(D) \geq 2 \textrm{ even} \}$$
due to Proposition \ref{kempf}. This means that $\mult_{\widetilde{\theta}}(x)$ is even for all $x \in \widetilde{\theta}|_{P_0}$. In particular, this holds for all generic points $x$ of $\widetilde{\theta}|_{P_0}$, and so indeed $\widetilde{\theta}|_{P_0} = 2D$ for some effective divisor $D$ on $P_0$.
\end{proof}

\begin{definition}\label{xi}
Define the \emph{theta divisor} $\Xi \subset P^+$ as $D-\pi^*\mathcal{L}_0$. Then $\lambda_{\Xi}$ is the principal polarisation $\lambda_D$ on $P^+$ constructed in Proposition \ref{algequiv}.
\end{definition}

\begin{corollary}\label{descSing}
The singularities of $\Xi + \pi^*\mathcal{L}_0$ are given by 
\begin{align*}\{x \in P_0 : \mult_{\widetilde{\theta}}(x) \geq 4 \} \cup \{ & x \in P_0 : \mult_{\widetilde{\theta}}(x) = 2 \textrm{ and } \\ & T_{x,P_0} \subset \textrm{ tangent cone to $\widetilde{\theta}$ at $x$}\}.\end{align*}
\end{corollary}

\section{Cubic threefolds}
In this section we study cubic threefolds over an algebraically closed field $k$. We classify them into two categories, based on whether or not they contain a good line. We show that a cubic threefold has no good lines if and only if they are Hermitian and $\textrm{char}(k) = 2$. It turns out that there is only one isomorphism class of Hermitian cubic threefolds, represented by the Fermat cubic threefold. For non-Hermitian cubic threefolds, we can construct a conic bundle and a respective Prym variety. In this case we show that the origin is the sole singularity of the theta divisor on the Prym variety. 

\subsection{Lines on cubic threefolds}
In this section we study the lines on a cubic threefold. We explain how to give the structure of a conic bundle to cubic threefolds with a good line, after some birational modification.

Let $X = \mathbb{V}_+(f)$ be a smooth cubic threefold. Let $F(X)$ be the Fano scheme of lines on $X$, which is a closed subscheme of the Grassmannian $\mathrm{Gr}(4,2)$.

\begin{proposition}\label{Fsurface}
$F(X)$ is smooth, connected and $2$-dimensional.
\end{proposition}
\begin{proof}
See \cite[Theorem 1.3]{altman} for the fact that $F(X)$ is $2$-dimensional, and \cite[Corollary 1.12]{altman} for smoothness. For connectedness see \cite[Proposition 1.15(i)]{altman}.
\end{proof}

Since $F(X)(k) \neq \emptyset$, there must be a line on $X$. We can assume, after linear change of coordinates, that $X$ contains the line $l = \mathbb{V}_+(x_0,x_1,x_2)$. In this case we may write
$$f(x_0,x_1,x_2,x_3,x_4) = x_3^2L_0 + x_3x_4L_1 + x_4^2L_2 + x_3Q_0 + x_4Q_1 + R$$
for linear $L_0,L_1,L_2$, quadratic $Q_0, Q_1$ and cubic $R$ in $x_0,x_1,x_2$.

\begin{definition}
Consider the following two subschemes of $F(X)$
$$F_0(X) := \{ l \in F  :  \textrm{there exists a } 2\textrm{-plane } P\subset \mathbb{P}^4_{k(l)}  \textrm{ s.t.} P\cap X = 2l \cup l' \textrm{ for some } l'\},$$
$$F_1(X) := \{ l \in F  :  \textrm{there exists a } 2\textrm{-plane } P\subset \mathbb{P}^4_{k(l)}  \textrm{ s.t.} P\cap X = 2l' \cup l \textrm{ for some } l'\}.$$
we call a line $l \in F(X)$ a \emph{good line} if $l \not \in F_0(X) \cup F_1(X)$.
\end{definition}

\begin{lemma}\label{F0dimension}
The schemes $F_0(X)$ and $F_1(X)$ are Zariski closed subschemes of $F(X)$. Moreover, $\dim(F_0(X)) \geq \dim(F_1(X))$.
\end{lemma}
\begin{proof}
See \cite[Lemma 1.4]{murre2} for the fact that $F_0(X) \subset F(X)$ is Zariski closed. The proof is independent of the choice of characteristic.

There is a surjective map $F_0(X) \to F_1(X)$ sending a line $l$ to the residual line in the intersection $P \cap X$ when $P$ is the tangent plane to $X$ at $l$. This is well-defined since $P$ is unique. This map is algebraic and therefore its image $F_1(X)$ is Zariski closed, and $\dim(F_0(X)) \geq \dim(F_1(X))$.
\end{proof}

\begin{remark}\label{goodlinebasis}
Suppose that the line $\mathbb{V}_+(x_0,x_1,x_2)$ on $X$ is not in $F_0(X)(k)$. From the proof of \cite[Lemma 1.4]{murre2}, this is equivalent to $L_0,L_1,L_2$ being linearly independent in $x_0,x_1,x_2$. In this case, by linear coordinate change we may assume that $L_i = x_i$ for all $i$. 
\end{remark}

\begin{proposition}\label{formulaDelta}
Suppose that $l = \mathbb{V}_+(x_0,x_1,x_2)$ is a good line on $X$ and write
$$f(x_0,x_1,x_2,u,v) = u^2x_0 + uvx_1 + v^2x_2 + uQ_0 + vQ_1 + R$$
as in Remark \ref{goodlinebasis}. Let $\pi : X\backslash l \to \mathbb{P}^2_k$ be the induced projection onto $\mathbb{V}_+(u,v)$. This endows $\Bl_l(X)$ with a conic bundle structure. Let $[y_0: y_1: y_2]$ be coordinates on $\mathbb{P}^2_k$. Then the discriminant curve $C$ is a smooth degree $5$ plane curve.
\end{proposition}
\begin{proof}
If $\textrm{char}(k) \neq 2$ then this result can be found in \cite[Proposition 1.22]{murre2}. Suppose that $\textrm{char}(k) = 2$.

    Consider the $2$-planes passing through $l$. They are parametrized by coordinates $T = [y_0:y_1:y_2]$ corresponding to a plane $P_T$ spanned by $T$ and $l$. Then $P_T \cap X = K_T \cup l$ where $K_T$ is the conic given by the vanishing of 
$$G_T(u,v,t) = u^2y_0 + uvy_1 + v^2y_2 + utQ_0(y_0,y_1,y_2) + vtQ_1(y_0,y_1,y_2) + t^2R(y_0,y_1,y_2).$$
The discriminant of $G_T(u,v,t)$ is given by
$$H(y_0, y_1, y_2) = y_0 Q_1^2 + y_1^2R + y_1 Q_0Q_1 + y_2 Q_0^2.$$
Since $K_T$ is the compactification of the fiber $\pi^{-1}(T)$, then $C := \mathbb{V}_+(H)$ is the discriminant locus of the conic bundle. The Jacobian matrix of $H$ is 
$$JH = \begin{pmatrix}
Q_1^2 + y_1 (\partial_{y_0}Q_0)Q_1 + y_1 Q_0(\partial_{y_0}Q_1) + y_1^2(\partial_{y_0}R) \\
Q_0Q_1 + y_1 (\partial_{y_1}Q_0)Q_1 + y_1 Q_0(\partial_{y_1}Q_1) + y_1^2(\partial_{y_1}R) \\
Q_0^2 + y_1 (\partial_{y_2}Q_0)Q_1 + y_1 Q_0(\partial_{y_2}Q_1) + y_1^2(\partial_{y_2}R)
\end{pmatrix}.$$
Suppose that $[y_0:y_1:y_2]$ is a singular point of $C$. The Jacobian matrix of $f(x_0,x_1,x_2,u,v)$ is
$$Jf = \begin{pmatrix}
u^2 + u(\partial_{x_0}Q_0) + v(\partial_{x_0}Q_1)+ (\partial_{x_0}R) \\
uv + u(\partial_{x_1}Q_0) + v(\partial_{x_1}Q_1)+ (\partial_{x_1}R) \\
v^2 + u(\partial_{x_2}Q_0) + v(\partial_{x_2}Q_1)+ (\partial_{x_2}R) \\
vx_1 + Q_0 \\
ux_1 + Q_1
\end{pmatrix}$$
so we observe that
$$p = (y_1y_0,y_1^2,y_1y_2,Q_1(y_0,y_1,y_2),Q_0(y_0,y_1,y_2))$$
is a singular point on $X$ if $y_1 \neq 0$. As a result $y_1 = 0$. From the vanishing of $JH$ we must have
$$Q_0^2 = Q_0Q_1 = Q_1^2 = 0$$
which means that $Q_0 = Q_1 = 0$. The corresponding conic in $\mathbb{P}^2_k$ becomes
$$K_{[y_0:0:y_2]} = \mathbb{V}_+(u^2y_0 + v^2y_2 + t^2R(y_0,0,y_2))$$
which is a double line. However $l \not \in F_1$ so we get a contradiction. Thus $C$ is a smooth plane curve of degree $5$.
\end{proof}

\begin{remark}\label{change2plane}
If we fix a line $l \subset X$ in a smooth cubic threefold but vary the $2$-plane $P$ we are projecting onto, this amounts to a linear change of coordinates for the discriminant curve $C$. Therefore the isomorphism class of the discriminant curve $C$ is well-defined given a pair $(X,l)$.
\end{remark}

\begin{definition}[Double cover of curves attached to a cubic threefold equipped with a line]\label{doublecoverattached}
Let $X$ be a smooth cubic threefold with a line $l \subset X$. Let $\pi: Bl_l(X) \to \mathbb{P}^2_k$ be the conic bundle induced by projecting away from $l$ onto a $2$-plane $P \subset \mathbb{P}^4_k$ satisfying $P \cap l = \emptyset$. Let $C$ be the discriminant curve for $\pi$. Consider the Zariski closed subset
$$\widetilde{C} := \{l' \in F(X) : l \textrm{ intersects every line in } \overline{l'}\} \subset F(X).$$ 
Then $\widetilde{C}$ parametrizes the lines on $X$ which lie in a fiber of $\pi$, so there is a natural morphism $\tau: \widetilde{C} \to C$. We call this the \emph{double cover attached to the pair} $(X,l)$. Clearly, the isomorphism class of $\tau$ is independent of the $2$-plane $P$ chosen.
\end{definition}

\begin{proposition}
Let $X$ be a smooth cubic threefold over a field $k$ and let $l\subset X$ be a line. Then $l$ is a good line if and only if the discriminant curve $C$ is smooth, and the double cover attached to $(X,l)$ is \'etale. 
\end{proposition}
\begin{proof}
Suppose that $l \subset X$ is a good line. We may apply a linear change of coordinates so that $l = \mathbb{V}_+(x_0,x_1,x_2)$. By Proposition \ref{formulaDelta}, the discriminant curve given by projecting onto $P = \mathbb{V}_+(u,v)$ is smooth. Furthermore, since $l \not \in F_0(X)(k) \cup F_1(X)(k)$, the fibers of the double cover $\tau:\widetilde{C} \to C$ attached to $(X,l)$ are all given by $2$ distinct points, and so $\tau$ is finite and unramified. The scheme $\widetilde{C}$ can be easily shown to be integral, and $C$ is a Dedekind scheme. Therefore $\tau$ is flat, and thus \'etale. 

Suppose that the pair $(X,l)$ is chosen so that the discriminant curve $C$ is smooth and the double cover attached to it is \'etale. Then it follows by definition that $l \not \in  F_0(X)(k) \cup F_1(X)(k)$.
\end{proof}

\begin{example}
We show that smooth cubic threefolds with a good line exist over fields of characteristic $2$. In fact we construct an example over $\mathbb{F}_2$. Consider the cubic form
$$f = x_3^2x_0 + x_3x_4x_1 + x_4^2x_2 + x_3(x_0^2 + x_1^2) + x_4(x_1^2 + x_2^2) + x_0x_2^2 + x_2x_0^2$$
which contains the line $l = \mathbb{V}_+(x_0,x_1,x_2)$. It has Jacobian matrix
$$Jf = \begin{pmatrix}
x_3^2 + x_2^2 \\
x_3x_4 \\
x_4^2 + x_0^2 \\
x_4x_1 + x_0^2 + x_1^2 \\
x_3x_1 + x_1^2 + x_2^2
\end{pmatrix}$$
and a simple analysis shows that $X = \mathbb{V}_+(f)$ is smooth. Clearly $l \not \in F_0(X)(k)$ as the linear terms $L_0,L_1,L_2$ are linearly independent, as per Remark \ref{goodlinebasis}. Finally, $l \not \in F_1(X)(k)$.
\end{example}

\subsection{Cubic threefolds with a good line}
The aim of this subsection is to show that the theta divisor on a Prym variety attached to a cubic threefold with a good line has a sole singularity, given by the origin. In this subsection all cubic threefolds are assumed to contain a good line.

\begin{definition}(Prym variety associated to a cubic threefold with a good line)\label{prymCubic}
Let $X$ be a smooth cubic threefold with a good line $l \subset X$. Let $\pi: \widetilde{C} \to C$ be the double cover attached to $(X,l)$. We call
$$\Prym(X,l) := \Prym(\widetilde{C}/C)$$
the \emph{Prym variety associated to the pair} $(X,l)$.
\end{definition}

 Let $\pi: \widetilde{C} \to C$ be an \'etale double cover associated to a smooth cubic threefold $X$ with a good line $l$ as in Definition \ref{doublecoverattached}, and let $\sigma: \widetilde{C} \to \widetilde{C}$ be the non-trivial involution above $C$. Let $P^+$ the the associated Prym variety. Then $g(C) = 6$ and $g(\widetilde{C}) = 11$, so we have $\dim(P^+) = 5$. Let $H$ be the hyperplane section on $C$. Then $H^2 \cong K_C$, and $\pi^*H$ is a theta characteristic for $\widetilde{C}$. The main purpose of this subsection is to show that $\pi^*H$ is an even theta characteristic. 

\begin{lemma}\label{preH}
Consider the morphism $g:\widetilde{C} \to l$ sending a line $l' \in \widetilde{C}(k)$ to its intersection point with $l$. Given $p \in l(k)$, let $L := g^*(p)$. Then $\pi_*L \cong H$ and $|L|$ is a basepoint-free linear system of degree $5$.
\end{lemma}
\begin{proof}
This follows from \cite[2.(ii)]{beau}, whose proof does not make use of the characteristic of the base field.
\end{proof}

\begin{corollary}\label{tildeConnected}
The curve $\widetilde{C}$ is connected.
\end{corollary}
\begin{proof}
Suppose that $\widetilde{C}$ is not connected. Then $\widetilde{C} = C \cup C$ and the map $g: \widetilde{C} \to l$ from Lemma \ref{preH} restricts to some morphism $g':C \to \mathbb{P}^1_k$ of degree $\deg(g') \leq 2$. We know that $C$ is not rational. Since $K_C$ is very ample, it follows from \cite[Proposition \rom{4}.5.2]{hartshorne} that $C$ is not hyperelliptic. Thus any non-constant morphism from $C$ to $\mathbb{P}^1_k$ has degree greater than $2$ and we get a contradiction.
\end{proof} 

\begin{lemma}\label{cupext}
Consider a short exact sequence 
\begin{align}0 \to \mathcal{O}_C \to V \to \mathcal{O}_C \to 0\end{align}
on a curve $C$, and a flat $\mathcal{O}_C$-module $W$. Then the connecting homomorphism for the long exact sequence attached to 
$$0 \to \mathcal{O}_C\otimes W \to V\otimes W \to \mathcal{O}_C \otimes W \to 0$$ 
is given by cup product with the extension class of $(4.1)$ as an element of $\HH^1(\mathcal{O}_C)$.
\end{lemma}
\begin{proof}
By \cite[Proposition \rom{3}.6.3]{hartshorne}, the functors $\Hom(\mathcal{O}_C, \cdot)$ and $\Gamma(C, \cdot)$ are equal so $\Ext^i(\mathcal{O}_C, \mathcal{O}_C) = \HH^i(\mathcal{O}_C)$ for all $i$. By \cite[Exercise \rom{3}.6.1]{hartshorne}, the extensions 
$$0 \to \mathcal{O}_C \to V \to \mathcal{O}_C \to 0$$
are classified by $\Ext^1(\mathcal{O}_C, \mathcal{O}_C) = \HH^1(\mathcal{O}_C)$. This particular extension corresponds to $s = \delta(1)$ if $\delta:  \HH^0(\mathcal{O}_C) \to  \HH^1(\mathcal{O}_C)$ is the connecting homomorphism in the long exact sequence. By \cite[Theorem \rom{2}.7.1(b)]{bredon} if
$$0 \to \mathcal{O}_C\otimes W \to V\otimes W \to \mathcal{O}_C \otimes W \to 0$$ 
is exact then $\delta(a \cup b) = \delta(a) \cup b$ for all $a \in  \HH^0(\mathcal{O}_C), b \in  \HH^0(W)$. It follows that the connecting homomorphism for the long exact sequence is given by cup product with $s$.
\end{proof}

\begin{lemma}\label{upper}
We have the bound $h^0(\mathcal{O}(\pi^*(H))) \leq 4$.
\end{lemma}
\begin{proof}
Suppose that $\textrm{char}(k) \neq 2$. Then this result follows from \cite[2.(iv)]{beau}.

Now suppose that $\textrm{char}(k) = 2$. Consider the structure morphism $\pi^{\#}: \mathcal{O}_{C} \to \pi_*\mathcal{O}_{\widetilde{C}}$. Affine locally, this is given by an Artin-Schreier extension since $\pi$ is \'etale of degree $2 = \textrm{char}(k)$. Therefore the cocycle data for $\pi_*\mathcal{O}_{\widetilde{C}}$ is unipotent and we have a short exact sequence
\begin{align}0 \to \mathcal{O}_C \xrightarrow{\pi^{\#}} \pi_*\mathcal{O}_{\widetilde{C}} \to \mathcal{O}_C \to 0.\end{align}
Since $\mathcal{O}(H)$ is flat, we get
$$0 \to \mathcal{O}_C(H) \to (\pi_*\mathcal{O}_{\widetilde{C}})(H) \to \mathcal{O}_C(H) \to 0$$
but $\pi$ is \'etale and so
$$\HH^0(\mathcal{O}(\pi^*H)) = \HH^0(\pi_*\mathcal{O}(\pi^*H)) = \HH^0(\pi_*\pi^*\mathcal{O}(H)) = \HH^0(\mathcal{O}(H)\otimes \pi_*\mathcal{O}_{\widetilde{C}})$$
which means that we are interested in the second term of the long exact sequence
\begin{center}\begin{tikzpicture}[descr/.style={fill=white,inner sep=1.5pt}]
        \matrix (m) [
            matrix of math nodes,
            row sep=1em,
            column sep=2.5em,
            text height=1.5ex, text depth=0.25ex
        ]
        { 0 &  \HH^0(\mathcal{O}_C(H)) & \HH^0((\pi_*\mathcal{O}_{\widetilde{C}})(H)) & \HH^0(\mathcal{O}_C(H)) \\
            & \HH^1(\mathcal{O}_C(H)) & \HH^1((\pi_*\mathcal{O}_{\widetilde{C}})(H)) &\HH^1(\mathcal{O}_C(H)) & 0.\\
        };

        \path[overlay,->, font=\scriptsize,>=latex]
        (m-1-1) edge (m-1-2)
        (m-1-2) edge (m-1-3)
        (m-1-3) edge (m-1-4)
        (m-1-4) edge[out=355,in=175] node[descr,yshift=0.3ex] {$\delta$} (m-2-2)
        (m-2-2) edge (m-2-3)
        (m-2-3) edge (m-2-4)
        (m-2-4) edge (m-2-5);
\end{tikzpicture}\end{center}
Since $h^0(\mathcal{O}_C(H)) = 3$ it suffices to prove that $\rank(\delta) \geq 2$. 

By Lemma \ref{cupext}, the connecting homomorphism $\delta$ is given by cup product by the extension class $(4.2)$. Using Serre duality we can also represent $\delta$ as a map $ \HH^0(\mathcal{O}_C(H)) \to \HH^0(\mathcal{O}_C(H))^{\vee}$, or alternatively as a quadratic form 
$$s:  \HH^0(\mathcal{O}_C(H)) \otimes \HH^0(\mathcal{O}_C(H)) \to k.$$
It suffices to show that $\rank(s) \geq 2$. If $\rank(s) = 0$, then $\delta$ is the zero map, so $(4.2)$ must be the trivial extension. This means $\widetilde{C} \cong C \cup C$, so we get a contradiction by Corollary \ref{tildeConnected}. 

Suppose now that $\rank(s) = 1$. Let $\langle x,y,z \rangle$ be a basis for $\HH^0(\mathcal{O}_C(H))$. We may assume after linear change of variables that
$$s = x^2.$$
The set $\HH^1(C,\mathbb{F}_2)$ classifies $\mathbb{Z}/2\mathbb{Z}-$covers of $C$. The Artin-Schreier sequence
$$\HH^1(C,\mathbb{F}_2) \to \HH^1(\mathcal{O}_C) \xrightarrow{F - 1} \HH^1(\mathcal{O}_C)$$
sends an element $[\pi: \widetilde{C} \to C]$ of $\HH^1(C,\mathbb{F}_2)$ to the extension class $(4.2)$. It follows that the extension class $(4.2)$ in $\HH^1(\mathcal{O}_C)$ is invariant under the Frobenius map. The Frobenius map is dual to the Cartier operator
 $$\mathcal{C}: \HH^0(\Omega_C) \to \HH^0(\Omega_C)$$
under Serre duality. Therefore $s$ must be invariant under $\mathcal{C}$ when realized as an element of $\HH^0(\mathcal{O}_C(2H))^{\vee}$.

Let $F(x,y)$ be the degree $5$ equation defining $C$ on a standard affine chart of $\mathbb{P}^2$, and $F_y := \frac{\partial F}{\partial y}$. Then by \cite[2.3]{costa} we obtain a basis for $\HH^0(\Omega_C)$ given by
$$\left \{\omega_{kl} = x^{k-1}y^{l-1}\frac{dx}{F_y} : k,l \geq 1, k+l \leq 4 \right \}.$$
Write $F(x,y) = \sum_{i,j}F_{ij}x^iy^j$. Then by \cite[2.5]{costa}, the Cartier operator acts as $\mathcal{C}(\omega_{kl}) = \sum_{i,j \geq 1}F_{2i - k, 2j - l}^{1/2}\omega_{ij}$. Under the isomorphism $\HH^0(\Omega_C) \cong \HH^0(\mathcal{O}_C(2))$, the differential forms $\omega_{kl}$ correspond to $x^{k-1}y^{l-1}z^{4 - (k+l)}$, and $s$ is the function which selects the $x^2$ coefficient. For $s$ to be $\mathcal{C}$-invariant, we must have 
$$s\left (\mathcal{C} \left (\sum_{i,j}A_{i,j}x^iy^j \right ) \right ) = \sum_{i,j}F_{5-i,1-j}^{1/2}A_{i,j} = s\left (\sum_{i,j}A_{i,j}x^iy^j \right ) = A_{2,0}$$
for all possible coefficients $A_{i,j}$. Consider the monomials $x^iy^j$ for $(i,j) = \newline (0,1), (1,0), (1,1)$. These cases respectively tell us that $F_{5,0} = F_{4,1} = F_{4,0} = 0$. Then $[1:0:0]$ must be a singular point on $C$ which gives us a contradiction. This shows that $\rank(s) \geq 2$ and so $h^0(\pi^*(H)) \leq 4$.
\end{proof}

\begin{proposition}\label{secH}
We have $h^0(\mathcal{O}(\pi^*H)) = 4$ and so $[\pi^*H] \in P_0(k)$.
\end{proposition}
\begin{proof}
This proof is based on \cite[2.(iv)]{beau}. We write down the details of the proof in order to show the reader that the proof is valid in all characteristics. Note that $\pi^*H = \sigma L + L$. We know $h^0(\sigma L + L) \leq 4$ by Lemma \ref{upper}, and that $h^0(\sigma L + L) \geq 3$ since $\HH^0(C, \mathcal{O}_C(H)) \subset \HH^0(\sigma L + L)$. All sections coming from $\HH^0(C, \mathcal{O}_C(H))$ are invariant under the action of $\sigma$, so it suffices to find a global section which is not invariant under $\sigma$. 

 Let $L$ be chosen as in Lemma \ref{preH} and choose sections $s,t$ of $\mathcal{O}(L)$ whose divisors do not share any common points. This is possible, for example, by pulling back a basis of the global sections of $\mathcal{O}_{\mathbb{P}^1}(1)$ along $g: \widetilde{C} \to l$. Then $L$ must be globally generated by $s,t$ and we have the short exact sequence 
    $$0 \xrightarrow{} \mathcal{O}(\sigma L-L) \xrightarrow{(t, -s)} \mathcal{O}(\sigma L)^2 \xrightarrow{(s,t)^T} \mathcal{O}(\sigma L +L) \xrightarrow{} 0.$$
Now $h^0(\sigma L - L) = 0$ or $1$ because the only possible effective divisor linearly equivalent to $\sigma L - L$ is $0$. Also $h^0(\sigma L) = h^0(L) \geq 2$ since $s,t \in h^0(L)$. By looking at the long exact sequence in sheaf cohomology, then in fact we get $h^0(\sigma L) = h^0(L) = 2$. Suppose that $h^0(\sigma L + L) = 3$. In this case, $h^0(\sigma L - L) = 1$ so we have a nonzero section $y \in \HH^0(\sigma L - L)$ such that $\Div(y) = \sigma L - L$. However $\sigma(y) = y$ must also hold, which would imply that $\sigma(\sigma L - L) = \sigma L - L$. This in turn implies that $\deg(L)$ must be even, which is a contradiction as $\deg(L) = 5$. Therefore $h^0(\sigma L + L) = 4$ as needed.
\end{proof}

\begin{definition}
Consider the divisor $\Xi + [\pi^*H]$ on $P_0$. The \emph{special singularities} $\Sing_{sp}(\Xi + [\pi^*H])$ are defined as the set
$$ \{x \in P_0(k) : \mult_{\widetilde{\theta}}(x) = 2 \textrm{ and } T_{x,P_0} \subset \textrm{ tangent cone to $\widetilde{\theta}$ at $x$}\}.$$
The \emph{stable singularities} $\Sing_{st}(\Xi + [\pi^*H])$ are defined as the set
$$ \{x \in P_0(k) : \mult_{\widetilde{\theta}}(x) \geq 4 \}.$$
\end{definition}

\begin{lemma}\label{salvage}
Let $x \in \Sing_{sp}(\Xi + [\pi^*H])$ be a special singularity. Let $L_x$ be the associated line bundle. Then the pairing
    $$\langle, \rangle : \HH^0(L_x) \times \HH^0(L_x) \to \HH^0(\Omega_{\widetilde{C}}^1) \cong (T_xJ_{10}\widetilde{C})^{\vee}$$
    sending $(a,b) \mapsto a \otimes \sigma^* b$ is symmetric.
\end{lemma}
\begin{proof}
    We have $h^0(L_x) = 2$ by the Riemann-Kempf theorem, so we can choose a basis $\{ s, t \}$ for $H^0(L_x)$. By \cite[Theorem 2]{gk}, the tangent cone to $\widetilde{\theta}$ at $x$ inside $T_xJ_{10}\widetilde{C}$ is described by $\det(W) = 0$ where 
    $$W = \begin{pmatrix}
\langle s, s \rangle & \langle t, s \rangle \\
\langle s, t \rangle & \langle t, t \rangle
\end{pmatrix}.$$
We have an isomorphism $\HH^0(\Omega_{\widetilde{C}}^1) \cong (T_0J\widetilde{C})^{\vee}$ given by Serre duality and Proposition \ref{tspace}, and the action of $\sigma$ respects this isomorphism. Moreover, $(T_0J\widetilde{C})^{\vee} \cong (T_xJ_{10}\widetilde{C})^{\vee}$ via translation by $x$, and so $ (T_xJ_{10}\widetilde{C})^{\vee}$ attains an action by $\sigma$.

 The $6$-dimensional subspace $\HH^0(\Omega_C^1) \subset \HH^0(\Omega_{\widetilde{C}}^1)$ is a $\sigma$-invariant subspace. Note that by functoriality, the pullback of differential forms $\pi^* : \HH^0(\Omega_C^1) \to \HH^0(\Omega_{\widetilde{C}}^1)$ corresponds to pullback of cotangent spaces $\Nm^*: (T_{\Nm(x)}J_{10}C)^{\vee} \to (T_xJ_{10}\widetilde{C})^{\vee}$ under the norm map $\Nm$. The cotangents which lie in the image $\im(f)$ vanish on $T_xP_0 \subset T_xJ_{10}\widetilde{C}$. For dimension reasons this leads to a short exact sequence
 $$0 \to (T_{Nm(x)}J_{10}C)^{\vee} \xrightarrow{\Nm^*} (T_xJ_{10}\widetilde{C})^{\vee} \xrightarrow{i^*} (T_xP_0)^{\vee} \to 0.$$
Consider the pullback $i^*\det(W)$ of $\det(W)$ to $(T_xP_0)^{\vee} \wedge (T_xP_0)^{\vee}$. The vanishing of $i^*\det(W)$ describes the tangent cone to $\widetilde{\theta}|_{P_0}$ at $x$ inside $P_0$. Since $x$ is a special singulariy, then $T_{x,P_0}$ lies inside the tangent cone to $\widetilde{\theta}$ at $x$. This means that $i^*\det(W)$ must vanish on $T_xP_0$. 

The cotangents $\langle s, s \rangle$ and $\langle t, t \rangle$ lie in $(T_{\Nm(x)}JC)^{\vee}$ since they are norms of $s$ and $t$ respectively. As a result pulling the matrix $W$ back along $i$ yields
 $$i^*W = \begin{pmatrix}
0 & i^*\langle t, s \rangle \\
i^*\langle s, t \rangle & 0
\end{pmatrix}.$$
 Whose determinant must vanish. Thus either $i^*\langle s, t\rangle$ or $i^*\langle t, s\rangle$ must vanish. Suppose without loss of generality that $i^*\langle s,t \rangle = 0$. Then $\langle s, t \rangle$ lies in $(T_{\Nm(x)}JC)^{\vee}$ and is $\sigma$-invariant. Since $\sigma \langle s , t \rangle = \langle t, s \rangle$, we conclude that $\langle, \rangle$ is $\sigma$-invariant.
\end{proof}

\begin{corollary}\label{mumSS}
Let $L$ be a line bundle on $\widetilde{C}$ such that $h^0(L) = 2$ and $\Nm(L) = K_C$. Then $[L] \in P_0$ is a special singularity if and only if
$$L = \pi^*(M)\left (\sum_{i}x_i \right )$$
for some points $x_i \in \widetilde{C}$ and a line bundle $M$ on $C$ with $h^0(M) = 2$.
\end{corollary}
\begin{proof}
This follows from the proposition on \cite[page 343]{m-pv}, however we use Lemma \ref{salvage} as a subtitute for showing that the pairing $\langle, \rangle$ is symmetric.
\end{proof}

\begin{proposition}\label{sing}
The sole singularity of the theta divisor $\Xi \subset P_+$ is $0_{P_+}$.
\end{proposition}
\begin{proof}
This follows from \cite[Proposition 2]{beau}, however we use Corollary \ref{mumSS} when dealing with special singularities.
\end{proof}

\subsection{Cubic threefolds with no good lines}
In this subsection we will classify the cubic threefolds which have no good lines. It turns out that they can only exist in characteristic $2$, and there is a single isomorphism class over $k$ represented by the Fermat cubic threefold.

\begin{definition}
Let $X$ be a smooth projective variety with an embedding $i:X \to \mathbb{P}^n$. Let $(\mathbb{P}^n)^{\vee}$ be the dual projective space of hyperplanes. The \emph{Gauss map} $g:X \to (\mathbb{P}^n)^{\vee}$ is defined by sending $x \in X$ to the tangent space $T_xX$.
\end{definition}

\begin{definition}
Let $k$ be an arbitrary field of characteristic $2$. A \emph{Hermitian threefold} is a cubic hypersurface $\mathbb{V}_+(f) \subset \mathbb{P}^4_k$ such that $f$ contains no squarefree monomials.
\end{definition}

\begin{proposition}\label{2bic}
Let $k$ be an separably closed field of characteristic $2$. Then every smooth Hermitian threefold over $k$ is isomorphic to the Fermat cubic $X = \mathbb{V}_+(x_0^3 + x_1^3 + x_2^3 + x_3^3 + x_4^3)$.
\end{proposition}
\begin{proof}
This follows from \cite[Corollary 2.7]{qbic}.
\end{proof}

\begin{lemma}\label{NumberLines}
Let $X$ be a smooth cubic threefold and suppose that $\textrm{char}(k) = 2$. Let $\mathbb{L} \subset F(X) \times X$ be the universal family of lines on $X$, and $\pi: \mathbb{L} \to X$ the projection. Then $\pi$ is a generically finite morphism of degree $\deg(\pi) = 6$, and the separable degree of $\pi$ is either $3$ or $6$. 
\end{lemma}
\begin{proof}
This follows from \cite[Proposition 1.7]{altman}.
\end{proof}

\begin{lemma}\label{gaussLineInseparable}
Let $X$ be a smooth cubic threefold over $k$ such that $F_0(X) = F(X)$ and $\textrm{char}(k)=2$. Then for any line $L \subset X$, the map $g \big |_L$ is inseparable.
\end{lemma}
\begin{proof}
Let $L \subset X$ be a line. Then there is a $2$-plane $M$ tangent to $X$ at $L$, and $X \cap M = 2L \cup L'$ for some line $L'$. This means that the conclusion of \cite[Lemma 1.7]{furukawa} holds for all lines $L$ on $X$. Then the proofs of \cite[Proposition 1.6]{furukawa} and \cite[Proposition 1.5]{furukawa} follow for $X$, and \cite[Proposition 1.5]{furukawa} in particular shows that  $g \big |_L$ is inseparable for all lines $L \in F(X)$.
\end{proof}

\begin{proposition}\label{nogoodline}
Let $X$ be a smooth cubic threefold over $k$ with no good lines. Then $\textrm{char}(k) = 2$ and $X$ is isomorphic to the Fermat cubic threefold.
\end{proposition}
\begin{proof}
From the assumption, we have $F(X) = F_0(X) \cup F_1(X)$. However by Lemma \ref{F0dimension} we have $\dim(F_0(X)) \geq \dim(F_1(X))$. Since both $F_0(X), F_1(X)$ are Zariski closed subsets of $F(X)$, it follows that $F(X) = F_0(X)$.

Suppose that $\textrm{char}(k) \neq 2$. Then by \cite[Corollary 1.9]{murre2}, $F_0(X) = F(X)$ is a smooth curve, which gives us a contradiction.

Let $x \in X$ and select three distinct lines $L_1,L_2,L_3$ passing through $x$ using Lemma \ref{NumberLines}. Then $g \big |_{L_i}$ is inseparable for $i = 1,2,3$ by Lemma \ref{gaussLineInseparable}. Thus $\rank(\dd_x (g\big |_{L_i})) = 0$ for each $i$. This translates to $(\dd_xg)(T_xL_i) = 0$ for each $i$. Since the $L_1,L_2,L_3$ span at least a $2$-plane, we have $\rank(\dd_xg) \leq 1$. However $\dd_xg$ is a symmetric matrix with zeroes on the diagonal over a field of characteristic $2$. Therefore it has even rank, and so $\rank(\dd_xg) = 0$. It follows that the Hessian of the cubic form representing $X$ vanishes, and so $X$ is a Hermitian threefold. Therefore $X$ is isomorphic to the Fermat cubic threefold by Proposition \ref{2bic}.
\end{proof}

\begin{remark}
It is also easy to verify that if $X$ is a Hermitian threefold, then there are no good lines on $X$. However we do not need this result in this paper.
\end{remark}

\section{Intermediate Jacobians}
In this section we define intermediate Jacobians of cubic threefolds over arbitrary bases and construct them over arbitrary fields. We also prove that intermediate Jacobians are stable under smooth base change. Finally, we prove that the abelian variety underlying an intermediate Jacobian of a cubic threefold is isomorphic to the Albanese of its Fano variety of lines. We use this to prove a specialization result regarding algebraic representatives.

\subsection{Definition of intermediate Jacobians}
In this section we give various definitions of intermediate Jacobians and show that they agree. We also show that intermediate Jacobians, when they exist, satisfy Galois descent.

Let $k$ be an algebraically closed field and let $X$ be a cubic threefold over $k$.

\begin{definition}[Correspondences]
Let $D \subset A$ be a divisor on an abelian variety $A$. Consider the induced polarisation $\lambda_D : A \to A^{\vee}$ and let $y_A$ be the Poincar\'e bundle on $A \times A^{\vee}$. We define $J(D) := (1,\lambda_D)^*(y_A)$ as an element of $\Corr(A) := \Pic(A \times A)/(\pi_1^*\Pic(A) \oplus \pi_2^*\Pic(A))$, the group of correspondences of $A$.

Let $T$ be a smooth variety and let $z \in \Ch^2(X \times T)$. Consider the projection maps $\pi_1,\pi_2: X \times T \times T \to X \times T$ and $\pi: X\times T \times T \to T \times T$. Then we define 
$$I(z) := \pi_*(\pi_1^*z . \pi_2^*z) \in \Ch^1(T \times T).$$
\end{definition}

\begin{definition}\label{defARac}
Write $\Ch^n(X)$ for the Chow group of $n$-cycles on $X$, and $\A^n(X)$ for the subgroup of cycles algebraically equivalent to $0$. Given a smooth variety $T$, a mapping $f: T(k) \to \A^n(X)$ is called \emph{algebraic} if there is a cycle $Z \in \Ch^n(X \times T)$ such that $f(t) = Z_t$ for all $t \in T(k)$.

Let $A$ be an abelian variety. A homomorphism of abelian groups $g: \A^n(X) \to A(k)$ is called \emph{regular} if for all algebraic mappings $f: T(k) \to \A^n(X)$ from smooth varieties $T$, the map $g \circ f : T(k) \to A(k)$ is induced by a scheme morphism. We say $A$ is an \emph{algebraic representative} for $\A^n(X)$ if there is a regular homomorphism $g: \A^n(X) \to A(k)$ which is universal among all regular homomorphisms in the following sense. Given another regular homomorphism $g': \A^n(X) \to B(k)$, there exists a unique homomorphism of abelian varieties $h$ fitting into the commutative diagram
$$\begin{tikzcd}
 \A^n(X) \arrow{r}{g} \arrow{rd}{g'}  & A(k) \arrow[dotted]{d}{h} \\
  & B(k).
\end{tikzcd}
$$
We typically use the notation $\Ab^n_X$ for the $n^{\textrm{th}}$ algebraic representative of $X$. Given an algebraic representative $g: \A^2(X) \to A(k)$, we say a polarisation $\lambda_D: A \to A^{\vee}$ induced by a divisor $D$ is an \emph{incidence polarisation} if for all algebraic maps $f_{z}: T(k) \to \A^2(X)$ induced by $z \in \Ch^2(X \times T)$ we have
 $$(g \circ f_z)^*(J(D)) = -I(z).$$
\end{definition}

\begin{example}
Let $C$ be a curve. Let $\theta$ be the canonical theta divisor on $J_{g-1}C$, which induces a principal polarisation $\lambda: JC \to JC^{\vee}$. Then $JC$ is an algebraic representative for $\A^1(C)$, and $\lambda$ is an incidence polarisation.
\end{example}

\begin{proposition}\label{upincidence}
The incidence polarisation, if it exists, is unique.
\end{proposition}
\begin{proof}
This follows from \cite[Remark 3.4.3(i)]{beau2}. Although \cite{beau2} is under the assumption that $\textrm{char}(k) \neq 2$ for most of the paper, this particular result does not make use of this assumption.
\end{proof}

\begin{definition}\label{defIJac}
We define the \emph{intermediate Jacobian} of a cubic threefold $X$ as the pair $(\Ab^2_X,\lambda)$ whenever the algebraic representative $\Ab^2_X$ and the incidence polarisation $\lambda$ exist. We may also refer to this principally polarised abelian variety as $JX$.
\end{definition}

This defines intermediate Jacobians over algebraically closed fields. We now proceed to define intermediate Jacobians over arbitrary base schemes. 

\begin{definition}
A trivial family of cubic threefolds over a scheme $S$ is a subscheme of $\mathbb{P}_S^4$ flat over $S$ whose fibre above any $s \in S$ is a cubic threefold over $k(s)$. A family of cubic threefolds over a scheme $S$ is a morphism $X \to S$ such that there exists an \'etale covering $\{U_i \to S\}_i$ so that for each $i$, the base change $X\times_S U_i \to U_i$ is  trivial family of cubic threefolds. 
\end{definition}

Let $X$ be a family of cubic threefolds over a base scheme $S$.

\begin{definition}
Let $\Lambda$ be an integral Noetherian scheme, and $S$ an integral scheme which is smooth, separated and of finite type over $\Lambda$. Let $ Sm_{\Lambda}/S$ be the category of schemes which are separated and dominant over $S$, and smooth and of finite type over $\Lambda$.
\end{definition}

\begin{definition}
There is a contravariant functor $\mathcal{A}^2_{X/S/\Lambda}: Sm_{\Lambda}/S \to \textrm{AbGp}$ to the category of abelian groups defined as follows. 
$$\mathcal{A}^2_{X/S/\Lambda}(T) = \{Z \in \Ch^2(X_T) : Z_{\eta^{sep}_T} \sim_{alg} 0\}$$ 
where $\eta^{sep}_T$ is the separable closure of the generic point $\eta_T$ of $T$. For any morphism $f:T' \to T$ in $Sm_{\Lambda}/S$, we set $\mathcal{A}^2_{X/S/\Lambda}(f)$ to be the refined gysin homomorphism $f^{!}$. Here $\Lambda$ is simply a technical artifact, and we will omit it from the definition whenever we can take $S = \Lambda$.
\end{definition}

\begin{definition}
A \emph{regular morphism} is a natural transformation $h: \mathcal{A}^2_{X/S} \to A$ into some abelian scheme $A$ over $S$. We say $h$ is \emph{surjective} if $ \mathcal{A}^2_{X/S}(k) \to A(k)$ is surjective when $k$ is the separable closure of the generic point of $S$. An \emph{algebraic representative} $\Ab_{X/S}^2$ for $\mathcal{A}^2_{X/S}$ is a regular morphism which is initial among all regular morphisms.
\end{definition}

This functorial point of view extends the classical approach given in Definition \ref{defARac}. To be precise, if $S = \Spec(k)$ is the spectrum of an algebraically closed field, then in \cite[Section 1.4]{achter2022functorial} it is shown that the two approaches agree.

\begin{remark}\label{defsepclosed}
Definition \ref{defARac} can be extended to $k$ separably closed. This also agrees with the functorial definition of Achter et al.
\end{remark}

\begin{proposition}\label{existenceAR2}
Let $X$ be a family of cubic threefolds over a base scheme $S$ which is regular, integral and Noetherian. Then the algebraic representative $\Ab^2_{X/S/S}$ exists.
\end{proposition}
\begin{proof}
This follows from \cite[Theorem $7.7$]{achter2022functorial}.
\end{proof}

In particular, algebraic representatives always exist for cubic threefolds over fields. One of the key facts regarding algebraic representatives is the relation between the $\ell$-adic cohomology of a cubic threefold and its algebraic representative. Intimate knowledge of this relation is required to even define the intermediate Jacobian.

\begin{remark}
Algebraic representatives are stable under base change by separable field extensions due to \cite[Theorem 1]{achter2022functorial}. That is, if $X$ is a cubic threefold over a field $k$ and $L/k$ is a separable field extension, then $(\Ab^2_{X})_L \cong \Ab^2_{X_L}$.
\end{remark}

We postpone the proof of the following proposition. It is required in order to define the intermediate Jacobian.

\begin{proposition}\label{prereq}
Let $X$ be a cubic threefold over a field $k$, and let $\ell \not | \textrm{char}(k)$ be a prime.  Let $\phi: A^2(X_{k^{sep}}) \to \Ab^2_X(k^{sep}) $ be its algebraic representative and let $\lambda^2: T_{\ell}A(X_{k^{sep}}) \to \HH^3(X_{k^{sep}},\mathbb{Z}_{\ell}(2))$ be the $\ell$-adic Bloch map. Then $\phi$ is an isomorphism, which allows us to define the map $(\lambda^2 \circ T_{\ell}(\phi^{-1})): T_{\ell}\Ab^2_X(k^{sep}) \to\HH^3(X_{k^{sep}}, \mathbb{Z}_{\ell}(2))$. We also claim that the dual of this map defines an isomorphism $ \alpha: \HH^1(\Ab^2_{X_{k^{sep}}},\mathbb{Z}_{\ell}) \to \HH^3(X_{k^{sep}}, \mathbb{Z}_{\ell}(2))^{\vee}$.
\end{proposition}

\begin{definition}\cite[Property 2.4]{Benoist}\label{IJfield}
\noindent Let $X$ be a cubic threefold over a field $k$. A class $\theta \in \textrm{NS}(\Ab^2_{X})$ is \emph{distinguished} if it is a principal polarisation and for some prime $\ell \not | \textrm{char}(k)$, the image $cl_1(-\theta_{k^{sep}})$ under the first $\ell$-adic chern class
$$\textrm{NS}(\Ab^2_{X_{k^{sep}}}) \to \HH^2_{\textrm{\'et}}(\Ab^2_{X_{k^{sep}}},\mathbb{Z}_{\ell}(1)) \cong \left (\bigwedge^2 \HH^1_{\textrm{\'et}}(\Ab^2_{X_{k^{sep}}},\mathbb{Z}_{\ell}) \right )(1)$$
corresponds, under the isomorphism $\alpha$ from Proposition \ref{prereq}, to the cup product
$$\bigwedge^2 \HH^3_{\textrm{\'et}}(X_{k^{sep}},\mathbb{Z}_{\ell}(2)) \xrightarrow{\cup}  \HH^6_{\textrm{\'et}}(X_{k^{sep}},\mathbb{Z}_{\ell}(4)) \xrightarrow{\textrm{deg}} \mathbb{Z}_{\ell}(1).$$
If a distinguished class $\theta$ exists, then we call $(\Ab^2_X, \theta)$ the \emph{intermediate Jacobian} of $X$. One can check that this agrees with Definition \ref{defIJac}, and that any distinguished class $\theta$ is unique by the discussion following \cite[Property 2.4]{Benoist}.
\end{definition}

This defines intermediate Jacobians over arbitrary fields under some additional conditions, which we will confirm later on in this chapter. We are now ready to define intermediate Jacobians over arbitrary base schemes. 

\begin{definition}
Let $X$ be a family of cubic threefolds over a base scheme $S$. The intermediate Jacobian of $X$ is a principally polarised abelian scheme $(\Ab^2_{X/S},\Theta)$ such that $\Theta$ is flat over $S$ and the fibre $(\Ab^2_{X/S},\Theta)_s$ above any point $s \in S$ is the intermediate Jacobian of $X_s$. It is easy to show that such an object is unique, if it exists.
\end{definition}

We show that intermediate Jacobians over fields satisfy Galois descent.

\begin{lemma}\label{dist}
Let $D$ be an effective divisor on an abelian variety $A$ such that $\lambda_D$ is a principal polarisation. Then $D$ is the unique effective divisor in its linear equivalence class. Moreover, for any other effective divisor $E$ such that $\lambda_E = \lambda_D$, we have $E = D + x$ for some $x \in A$.
\end{lemma}
\begin{proof}
Let $L := \mathcal{O}(D)$. By \cite[Riemann-Roch Theorem]{m-av}, we obtain $\chi(L) = 1$. On the other hand, by \cite[Vanishing Theorem]{m-av} we have $H^i(A,L) = 0$ for all but one $i$ as $L$ is ample. Since $\mathbb{P}H^0(A,L)$ classifies effective divisors $E$ with $\mathcal{O}(E) \cong L$, and there is at least one such effective divisor $D$, we must have $h^0(A,L) = 1$. 

For the second part, let $E$ is an effective divisor. Then $\lambda_E = \lambda_D$ if and only if $E \sim_{alg} D$. Note that $A^{\vee} \cong \Pic^0(A)$ classifies the divisors algebraically equivalent to $0$ modulo linear equivalence. These are exactly the translation-invariant divisor classes. Then $[D] + \Pic^0(A)$ classifies the divisor classes which induce the polarisation $\lambda_D$. The map $\lambda_D: A \to \Pic^0(A)$ sending $x \mapsto [T_xD - D]$ is an isomorphism and so it follows that $\mathcal{O}(E) \cong T_x L $ for some $x \in A$. From he first part, it follows that $E = T_xD$. 
\end{proof}

\begin{proposition}\label{GalDescentIJ}
Let $X$ be a smooth cubic threefold over a field $k$, and let $L/k$ be a Galois extension. Suppose that the intermediate Jacobian $(\Ab^2_{X_L},\Theta)$ exists and that $\Theta$ can be represented by an effective divisor. Then the pair $(\Ab^2_{X_L},\Theta)$ is Galois-equivariant and descends to the intermediate Jacobian of $X$.
\end{proposition}
\begin{proof}
By \cite[Theorem 5.9]{achter2022functorial}, the algebraic representative $\Ab^2_X$ exists and is constructed via Galois descent. By \cite[Lemma 12.3]{achter2020decomposition}, the principal polarisation induced by $\Theta$ descends to a principal polarisation $\Lambda: \Ab^2_X \to (\Ab^2_X)^{\vee}$. Combined with Lemma \ref{dist}, we deduce that the effective divisor $\Theta$ is Galois-equivariant and descends to an effective divisor $\Theta_k$ on $\Ab^2_X$ inducing $\Lambda$. Moreover, $\Theta_k$ is distinguished by definition.
\end{proof}

\subsection{Smooth base change for intermediate Jacobians}
Let $X$ be a family of cubic threefolds over a base scheme $S$. In this section we collect some fact regarding algebraic representatives and prove that intermediate Jacobians are stable under base change by inverse limits of smooth morphisms when $S$ is smooth and connected over $\mathbb{Z}$.

\begin{proposition}\label{baseChangeCat}
Let $S' \to S$ be a morphism obtained as the inverse limit of morphisms in $Sm_{\Lambda}/S$. Then 
$$(\mathcal{A}^2_{X/S})_{S'} \cong \mathcal{A}^2_{X_{S'}/S'}$$
where the left hand side is obtained as a fibered product of functors.
\end{proposition}
\begin{proof}
This is equation $(2.2)$ in \cite[Section $2.1$]{achter2022functorial}.
\end{proof}

\begin{proposition}\label{extendRatMap}
Let $f:T \to A$ be a rational map of schemes over $S$, and suppose that $A$ is an abelian scheme and that $T$ is regular. Then $f$ extends uniquely to a morphism.
\end{proposition}
\begin{proof}
This is follows from \cite[Corollary 6, Section 8.4]{neron}.
\end{proof}

\begin{lemma}\label{liftfunctor}
Let $\eta$ be the generic point of $S$. Any (surjective) regular homomorphism $\phi_{\eta}:\mathcal{A}^2_{X_{\eta}/S_{\eta}} \to A_{\eta}$ lifts to a (surjective) regular homomorphism $\phi: \mathcal{A}^2_{X/S} \to A$.
\end{lemma}
\begin{proof}
We construct $\phi$ as follows. For each object $T$ in $Sm_{\Lambda}/S$, $\phi(T):\mathcal{A}^2_{X/S}(T) \to A(T)$ sends $Z \in \mathcal{A}^2_{X/S}(T) $ to the unique morphism $T \to A$ which lifts $\phi_{\eta}(T_{\eta})(Z_{\eta}): T_{\eta} \to A_{\eta}$. This makes sense because $\phi_{\eta}(T_{\eta})(Z_{\eta})$ is a map from a regular scheme to an abelian scheme, so it extends uniquely by Proposition \ref{extendRatMap}. By definition of surjective regular homomorphisms, if $\phi_{\eta}$ is surjective, so is $\phi$.
\end{proof}

We will now bundle up the results of Faltings-Chai and Vasiu on extensions of Abelian schemes into a single result. This is necessary in order to prove base change theorems regarding the algebraic representative.

\begin{definition}
Let $S$ be some scheme which is flat over a DVR $R$ with generic point $\eta$. We call $S$ \emph{healthy} if for any open $U \subset S$ with complement of codimension at least two such that $S_{\eta} \subset U$, and any abelian scheme $A$ over $U$, there is an abelian scheme $\widetilde{A}$ over $S$ so that $\widetilde{A}_U \cong A$. 
\end{definition}

\begin{proposition}\label{IJatP}
Let $S \to \mathbb{Z}_{(p)}$ be a smooth $\mathbb{Z}_{(p)}$-scheme where $p$ is any prime. Then $S$ is healthy.
\end{proposition}
\begin{proof}
From \cite[Theorem 1.3]{Vasiu}, we know that $S$ is healthy.
\end{proof}

\begin{proposition}\label{codim2extend}
Let $U \subset S$ be an open subset of a smooth connected scheme over $\mathbb{Z}$, with complement of codimension at least $2$. Let $A_U$ be an abelian scheme over $U$. Then $A_U$ extends to an abelian scheme $A$ over $S$.
\end{proposition}
\begin{proof}
Consider the poset of extensions $(U \subset W,A_W)$ of $A_U$. It suffices to show that the maximal element is defined over $S$. Take a maximal element $(U \subset W,A_W)$ and suppose $W \neq S$. Let $\eta \in S \backslash W$ be the generic point of an irreducible component of the complement. Since $S$ is smooth and $\overline{\eta}$ has codimension at least two, there is a chain $\overline{\eta} \subsetneq \overline{\zeta} \neq S$ for some point $\zeta \in S$.

Consider the stalk $S_{(\eta)}$. If this local ring is of mixed characteristic, then it is a healthy local ring by Proposition \ref{IJatP}. If it has residue characteristic $0$, then it is healthy by \cite[Corollary 6.8]{Faltings}. Since $A_W \big |_{S_{(\eta)}}$ is defined everywhere outside the maximal point, which has codimension at least $2$, it extends uniquely to an abelian scheme $A_{(\eta)}$ over $\Spec(S_{(\eta)})$.

We now proceed to spread $A_{(\eta)}$ out to an abelian scheme $A_V'$ defined over an open neighbourhood $V$ of $\eta$. This agrees with $A_W$ when pulled back to the stalk $S_{(\zeta)}$ for any point $\zeta$ of $S$ which specializes to $\eta$. The next step is to glue this abelian scheme with $A_W$, after possibly shrinking $V$, to a strictly larger abelian scheme.

Take the isomorphism $\phi_{\zeta}: (A_V')_{(\zeta)} \xrightarrow{\sim} (A_W)_{(\zeta)}$ of abelian schemes over the stalk at $\zeta$. This spreads out to an isomorphism of abelian schemes $\phi_{W'_{\zeta}}$ in some open subset $W'_{\zeta} \subset W \cap V$ which contains $\zeta$. Now let $\zeta$ vary across all points of $S$ which specialize to $\eta$, and consider $W' := \bigcup_{\zeta}W'_{\zeta} \subset W \cap V$. By removing some closed subsets, we can shrink $V$ so that $W \cap V = W'$, and $\eta$ still lies in $V$ since by construction, we cannot have removed any closed subset containing $\eta$.

Now we are in a position to glue to an abelian scheme $A_{V \cup W}$ over $V \cup W$. We just need to verify the cocycle conditions for $\{\phi_{W'_{\zeta}}\}_{\zeta}$. More precisely, for any $\zeta_1,\zeta_2$ we must verify
$$\phi_{W'_{\zeta_1}} \big |_{W'_{\zeta_1} \cap W'_{\zeta_2}} = \phi_{W'_{\zeta_2}} \big |_{W'_{\zeta_1} \cap W'_{\zeta_2}}.$$
If $\eta_S$ is the generic point of $S$, then actually $(\phi_{W'_{\zeta_1}})_{(\eta_S)} = (\phi_{W'_{\zeta_2}})_{(\eta_S)}$, and due to Proposition \ref{extendRatMap}, the two maps above are equal. This verifies the cocycle conditions and so we obtain an abelian scheme $A_{V \cup W}$ over $V \cup W$ extending $A_W$. We conclude that $W = S$ as required.
\end{proof}

\begin{remark}
Using results of \cite{Vasiu}, we can prove more general results about extensions of abelian schemes, but we do not need these results in this paper.
\end{remark}

\begin{lemma}\label{extendSurjRegHom}
Let $S$ be a smooth scheme over $\mathbb{Z}$ and let $X \to S$ be a family of cubic threefolds. Let $B$ be an abelian variety over $k(\eta_S)$ and consider a surjective regular homomorphism
$$\alpha: \mathcal{A}^2_{X_{\eta_S}/\eta_{S}} \to B.$$
Then $B$ extends to an abelian variety $B_S$ over $S$, and $\alpha$ extends to a surjective regular homomorphism $\widetilde{\alpha}: \mathcal{A}^2_{X_S/S}\to B_S$.
\end{lemma}
\begin{proof}
By \cite[Proposition 8.2]{achter2022functorial} there is a Galois-equivariant injection $T_{\ell}(B) \to \HH^3(X_{\eta_S},\mathbb{Z}_{\ell}(2))$ if $\ell \neq \textrm{char}(k(\eta_S))$ is coprime to some integer $r$. Let $\zeta$ be a divisor on $S$ and consider the DVR $S_{(\zeta)}$. By \cite[Lemma 6.1]{Achter_2017}, $B$ extends uniquely to an abelian scheme over $S_{(\zeta)}$. We can then spread this out to an abelian scheme $B_{V_{\zeta}}$ over an open neighbourhood $V_{\zeta}$ of $\zeta$. Note that $(B_{V_{\zeta}})_{\eta_S} \cong B$. As a result, for any two divisors $\zeta_1,\zeta_2$ there is an open $W_{\zeta_1,\zeta_2} \subset V_{\zeta_1} \cap V_{\zeta_2}$ and an isomorphism $\phi_{\zeta_1,\zeta_2}: (B_{V_{\zeta_1}})_{W_{\zeta_1,\zeta_2}} \xrightarrow{\sim}  (B_{V_{\zeta_2}})_{W_{\zeta_1,\zeta_2}}$. Due to Proposition \ref{extendRatMap}, these isomorphisms satisfy the cocycle conditions, once restricted to a common base. It follows that we can extend $B$ to an abelian scheme $B_V$ over an open $V \subset S$ whose complement has codimension at least $2$. Then Proposition \ref{codim2extend} extends $B_V$ to an abelian scheme $B_S$ over $S$. Finally, Lemma \ref{liftfunctor} allows us to lift $\alpha$ to a surjective regular homomorphism $\widetilde{\alpha}: \mathcal{A}^2_{X_S/S}\to B_S$.
\end{proof}

\begin{lemma}\label{monoSep}
Let $f:A \to B$ be a morphism of abelian varieties over a field $k$. Suppose that there is a morphism $g:B \to A$ such that $g \circ f = 1_A$, and that $f$ is surjective on $k^{sep}$-points. Then $f$ is an isomorphism.
\end{lemma}
\begin{proof}
The morphism $f$ is proper, hence quasi-compact. Since $B$ is smooth, it is geometrically reduced and so the generic point $\eta_B$ has separable residue field. A similar argument shows that the set of $k^{sep}$-points is dense in $A$ and $B$, so we conclude that $f^{sep}$ is surjective. The morphism $f^{sep}$ is also proper, and it is a monomorphism due to the existence of a retract $g^{sep}$. It is flat due to miracle flatness. It follows from \cite[Tag 06NC]{sp} that $f^{sep}$ is an isomorphism, whose inverse is $g^{sep}$. Then $f$ is also an isomorphism, with inverse $g$.
\end{proof}

\begin{proposition}\label{basechange}
Let $S$ be a smooth scheme over $\mathbb{Z}$, and let $X \to S$ be a family of cubic threefolds. Then for any morphism $S' \to S$ which is an inverse limit of morphisms in $Sm_{\mathbb{Z}}/S$, we have
$$(\Ab^2_{X/S})_{S'} \cong \Ab^2_{X_{S'}/S'}.$$
\end{proposition}
\begin{proof}
There is a natural map $g: \Ab^2_{X_{S'}/S'} \to (\Ab^2_{X/S})_{S'}$ due to Proposition \ref{baseChangeCat}. We need to show that this map is an isomorphism. By Proposition \ref{extendRatMap}, this occurs if and only if $g_{\eta_{S'}}: (\Ab^2_{X_{S'}/S'})_{\eta_{S'}} \to (\Ab^2_{X/S})_{\eta_{S'}}$ is an isomorphism, when $\eta_{S'}$ is the generic point of $S'$. Note that $k(\eta_{S'})$ is a separable field extension of $k(\eta_{S})$ and so by \cite[Theorem 5.10]{achter2022functorial} we have $(\Ab^2_{X_{\eta_S}/\eta_S})_{\eta_{S'}} \cong \Ab^2_{X_{\eta_{S'}}/\eta_{S'}}$ via the natural map. From the universal property we obtain morphisms
$$\begin{tikzcd}
\Ab^2_{X_{\eta_{S'}}/\eta_{S'}} \arrow{r}{h} \arrow{rd}{f} &(\Ab^2_{X_{S'}/S'})_{\eta_{S'}} \arrow{d}{g_{\eta_{S'}}} \\
   &   (\Ab^2_{X/S})_{\eta_{S'}}.
\end{tikzcd}$$
If $f$ and $h$ were isomorphisms, it would follow that $g_{\eta_{S'}}$ is an isomorphism. Therefore we are reduced to the case $S' = \eta_S$.

Consider the algebraic representative $\alpha: \mathcal{A}^2_{X_{\eta_S}/\eta_{S}} \to \Ab^2_{X_{\eta_S}/\eta_{S}}$. By Lemma \ref{extendSurjRegHom}, we can spread this out to $\widetilde{\alpha}: \mathcal{A}^2_{X_S/S}\to A$ for some abelian scheme $A$ over $S$. There is a natural morphism $j:\Ab^2_{X/S} \to A$. Consider the generic fiber
$$j_{\eta_S} :(\Ab^2_{X/S})_{\eta_S} \to A_{\eta_S} \cong \Ab^2_{X_{\eta_S}/\eta_{S}}.$$ 
By the universal property there is also a unique map $f:A_{\eta_S} \to (\Ab^2_{X/S})_{\eta_S}$. Then $j_{\eta_S} \circ f = 1_{ \Ab^2_{X_{\eta_S}/\eta_{S}}}$ since $\mathcal{A}^2_{X_{\eta_S}/\eta_{S}} \to \Ab^2_{X_{\eta_S}/\eta_{S}}$ is initial among natural transformations to abelian schemes. This means $f$ is injective.

Consider $\widetilde{\alpha}_{\eta_S}: \mathcal{A}^2_{X_{\eta_S}/\eta_{S}} \to (\Ab^2_{X_S/S})_{\eta_S}$. Then $f \circ j_{\eta_S} \circ \widetilde{\alpha}_{\eta_S} =\widetilde{\alpha}_{\eta_S}$ for the same reason as above, by lifting both morphisms to $\mathcal{A}^2_{X/S} \to \Ab^2_{X/S}$ using Lemma \ref{liftfunctor} and using the universal property. As $\widetilde{\alpha}$ is a surjective regular homomorphism, it is surjective on $\eta_S^{sep}$-points. Therefore $f$ is also surjective on $\eta_S^{sep}$-points and we conclude that $f$ is an isomorphism by Lemma \ref{monoSep}.
\end{proof}

This proves the smooth base change theorem for algebraic representatives. The next step is to consider what happens to the distinguished polarisation.

\begin{lemma}\label{cohfunc}
Let $X$ be a cubic threefold over a field $k$, and let $\ell \not | \textrm{char}(k)$ be a prime. Let $L/k$ be a field extension such that $(\Ab^2_X)_L \cong \Ab^2_{X_L/L}$. Suppose that
$$\lambda^2 \circ T_{\ell}(\phi^{-1}): T_\ell \Ab^2_{X_{k^{sep}}} \to \HH^3(X_{k^{sep}}, \mathbb{Z}_{\ell}(2))$$
is the map described in Proposition \ref{prereq}. Then this is functorial under the extension $i: k \to L$ in the sense that the following diagram commutes
$$\begin{tikzcd}
T_\ell \Ab^2_{X_{k^{sep}}} \arrow{r}{\lambda^2 \circ T_{\ell}(\phi^{-1})} \arrow{d}{i^*} & \HH^3(X_{k^{sep}}, \mathbb{Z}_{\ell}(2)) \arrow{d}{i^*} \\
  T_\ell \Ab^2_{X_{L^{sep}}} \arrow{r}{\lambda^2 \circ T_{\ell}(\phi^{-1})} & \HH^3(X_{L^{sep}}, \mathbb{Z}_{\ell}(2)).
\end{tikzcd}$$
Moreover, the vertical arrows are isomorphisms.
\end{lemma}
\begin{proof}
The second Bloch map $\lambda^2$ is functorial under flat pullbacks due to \cite[Proposition A.21]{achterBloch}. From the fact that $(\Ab^2_X)_L \cong \Ab^2_{X_L/L}$, we know that $\phi$ commutes with $i^*$. This proves that the above square is commutative. The right vertical arrows is an isomorphism by \cite[0A5E]{sp}, and the left vertical arrow is an isomorphism since all the $\ell^n$-torsion points of an abelian variety are defined over a separably closed field for all $n$.
\end{proof}

\begin{lemma}\label{bigCommDiagram}
Let $i: k \to l$ be an extension of separably closed fields, and let $\ell \not | \textrm{char}(k)$ be a prime. Let $X$ be a cubic threefold over $k$ so that $(\Ab^2_X)_l \cong \Ab^2_{X_l/l}$. Suppose that
$$\lambda^2 \circ T_{\ell}(\phi^{-1}): T_\ell \Ab^2_{X_{k^{sep}}} \to \HH^3(X_{k^{sep}}, \mathbb{Z}_{\ell}(2))$$
is the map described in Proposition \ref{prereq}. Then this induces the horizontal isomorphisms in the following diagram
$$\begin{tikzcd}
\HH^1(\Ab^2_{X_k},\mathbb{Z}_{\ell}) \arrow{r}{\alpha} \arrow{d}{i^*} & \HH^3(X_k, \mathbb{Z}_{\ell}(2))^{\vee}   \\
 \HH^1(\Ab^2_{X_l},\mathbb{Z}_{\ell}) \arrow{r}{\alpha} &  \HH^3(X_l, \mathbb{Z}_{\ell}(2))^{\vee} \arrow{u}{\_ \circ i^*}
\end{tikzcd}$$
which we claim is commutative.
\end{lemma}
\begin{proof}
Let us begin by recalling the construction of the perfect pairing $T_{\ell}(A) \times T_{\ell}(A^{\vee}) \to \mathbb{Z}_{\ell}$ for an abelian variety $A$. Take an $\ell^n$-torsion point $p \in A(k)$ and a line bundle $L \in A^{\vee}(k)$ with an isomorphism $\psi: L^{\ell^n} \to \mathcal{O}_A$. Then the pullback $T_p^*\psi : L^{\ell^n} \to \mathcal{O}_A$ by the translation by $p$ map is another isomorphism, and we get an element $\psi \circ (T_p^*\psi)^{-1} \in \Aut(\mathcal{O}_A)$. Note that $\Aut(\mathcal{O}_A) \cong \mathbb{G}_m$, and so this defines a pairing
$$A[\ell^n](k) \times A^{\vee}[\ell^n](k) \to \mathbb{Z}/\ell^n\mathbb{Z}$$
which is perfect. We then take the inverse limit of these pairings. From this construction, if $i: k \to l$ is an extension of separably closed fields, then the following diagram commutes
$$\begin{tikzcd}
T_\ell ((\Ab^2_{X_k})^{\vee}) \arrow{r}{\sim} \arrow{d}{i^*} &T_\ell (\Ab^2_{X_k})^{\vee}  \\
  T_\ell ((\Ab^2_{X_l})^{\vee}) \arrow{r}{\sim} &  T_\ell (\Ab^2_{X_l})^{\vee}  \arrow{u}{\_ \circ i^*}.
\end{tikzcd}$$
We get a commutative diagram
$$\begin{tikzcd}
\HH^1(\Ab^2_{X_k},\mathbb{Z}_{\ell}) \arrow{r} \arrow{d} & T_\ell ((\Ab^2_{X_k})^{\vee}) \arrow{r}{\sim} \arrow{d}{i^*} &T_\ell (\Ab^2_{X_k})^{\vee} \arrow{r}  & \HH^3(X_k, \mathbb{Z}_{\ell}(2))^{\vee}   \\
 \HH^1(\Ab^2_{X_l},\mathbb{Z}_{\ell}) \arrow{r} & T_\ell ((\Ab^2_{X_l})^{\vee}) \arrow{r}{\sim} &  T_\ell (\Ab^2_{X_l})^{\vee}  \arrow{u}{\_ \circ i^*} \arrow{r} & \HH^3(X_l, \mathbb{Z}_{\ell}(2))^{\vee} \arrow{u}
\end{tikzcd}$$
where the right square is achieved by applying duals to Lemma \ref{cohfunc}, and the left square arises by considering the long exact sequence attached to
$$0 \to \mu_{\ell^n} \to \mathbb{G}_m \xrightarrow{\ell^n} \mathbb{G}_m \to 0$$
over $\Ab^2_{X_k}$, and taking the inverse limit over $n$. The outer square commuting then finishes the proof.
\end{proof}

\begin{lemma}\label{extendPol}
Let $(\Ab^2_{X/k},\Theta)$ be the intermediate Jacobian of a cubic threefold $X$, and let $L/k$ be a field extension so that $(\Ab^2_X)_L \cong \Ab^2_{X_L/L}$. Then $(\Ab^2_{X/k},\Theta)_L \cong (\Ab^2_{X_L/L},\Theta_L)$ is the intermediate Jacobian of $X_L$.
\end{lemma}
\begin{proof}
We just need to show that $\Theta_L$ is distinguished. The first $\ell$-adic Chern class $cl_1$ is functorial with respect to flat pullbacks. Write $cl_1(\Theta) = a \cup b$ for $a,b \in \HH^1_{\textrm{\'et}}(\Ab^2_{X_{k^{sep}}},\mathbb{Z}_{\ell})$ and let $i:k \to L$ be the field extension. Then $cl_1(\Theta_L) = i^*cl_1(\Theta) = i^*a \cup i^*b$. Suppose that $a,b$ correspond to maps 
$$\widehat{a},\widehat{b}:  \HH^3(X_{k^{sep}}, \mathbb{Z}_{\ell}(2)) \to \mathbb{Z}_{\ell}$$
so that $\widehat{a} \cup \widehat{b}$ is the cup product map. Then by Lemma \ref{bigCommDiagram}, if $i^*a, i^*b$ correspond to maps
$$\widehat{i^*a},\widehat{i^*b}:  \HH^3(X_l, \mathbb{Z}_{\ell}(2)) \to \mathbb{Z}_{\ell}$$
 then we have $i^*(\widehat{i^*a}\cup \widehat{i^*b}) = \widehat{a} \cup \widehat{b}$ as the cup product map. This means that $\widehat{i^*a}\cup \widehat{i^*b}$ must be the cup product map, since cup product commutes with flat pullback. We conclude that $\Theta_L$ is distinguished.
\end{proof}

\begin{proposition}\label{baseChangeIJ}
Let $X \to S$ be a family of cubic threefolds and let $S' \to S$ be a morphism which is the inverse limit of smooth morphisms. Suppose that $(\Ab^2_{X/S})_{S'} \cong \Ab^2_{X_{S'}/S'}$ and that the intermediate Jacobian $(\Ab^2_{X/S},\Theta)$ of $X$ exists. Then $(\Ab^2_{X/S},\Theta)_{S'} = (\Ab^2_{X_{S'}/S'},\Theta_{S'})$ is the intermediate Jacobian of $X_{S'}$.
\end{proposition}
\begin{proof}
This follows from the definition of intermediate Jacobians and applying Lemma \ref{extendPol} to every fiber of $(\Ab^2_{X_{S'}/S'},\Theta_{S'})$.
\end{proof}

This reduces base change theorems for intermediate Jacobians to base change theorems for algebraic representatives.

\subsection{Intermediate Jacobians of non-Hermitian cubic threefolds}
In this chapter we show the existence of intermediate Jacobians for non-Hermitian cubic threefolds over fields, as well as for families of non-Hermitian cubic threefolds over normal Noetherian bases.

Let $X$ be a smooth cubic threefold with a good line $l$ over a separably closed field $k$. Let $\pi: \widetilde{C} \to C$ be the \'etale double cover attached to $(X,l)$ and let $\Prym(\widetilde{C}/C)$ be the Prym variety, with theta divisor $\Xi$. 

\begin{proposition}\label{algrep}
The map $\Ch^1(\widetilde{C}) \to \Ch^2(\Bl_l(X))$ sending the class of a point $x\in \widetilde{C}$ to its corresponding line $[L(x)] \in Ch^2(\Bl_l(X))$ induces an isomorphism
$$\Psi: \Prym(\widetilde{C}/C)(k) \cong A^1(\widetilde{C})/\pi^*A^1(C) \to A^2(\Bl_l(X)) \cong A^2(X)$$
which realizes $\Prym(\widetilde{C}/C)$ as an algebraic representative for $A^2(X)$.
\end{proposition}
\begin{proof}
We follow \cite[Proposition 3.3]{beau2} for the proof. All relevant details can be adapted to the case of separably closed fields. It is proven in \cite[3.3.2]{beau2} that $2\Psi$ is an algebraic map, and in \cite[Theorem 3.1]{beau2} that $\Psi$ is an isomorphism of groups. Let $g: A^2(X) \to B(k)$ be a regular morphism. Writing $P^+$ for $\Prym(\widetilde{C}/C)$, the map $g \circ 2\Psi : P^+(k) \to B(k)$ is induced by a morphism of abelian varieties. However we cannot directly conclude that $g \circ \Psi$ is a morphism of abelian varieties.

The morphism  $g \circ 2\Psi$ kills $P^+[2](k)$. Since $k$ is separably closed, the connected-\'etale sequence
$$0 \to P^+[2]^0 \to P^+[2] \to P^+[2](k) \to 0$$
splits and we can consider the map $q: P^+ \to P^+$ given by the quotient by $P^+[2](k)$. Note that $P^+[2](k)$ is killed under $g \circ 2\Psi$, so by passing to the quotient we may produce a morphism $g': P^+ \to B$ such that $g'\circ q = g \circ 2 \Psi$. On $k$-points this has the effect
$$\begin{tikzcd}
 A^2(X) \arrow{r}{g} \arrow{rd}{\Psi^{-1}}  & B(k) \\
  & P^+(k).  \arrow[dotted]{u}{g'}
\end{tikzcd}
$$
Two morphisms of abelian varieties coincide if and only if they coincide on $k$-points when $k$ is separably closed. As a result $g'$ must be the unique morphism of abelian varieties making this diagram commute, and so $P^+$ is an algebraic representative for $A^2(X)$.
\end{proof}

\begin{proposition}\label{incpol}
The principal polarisation $\lambda_{\Xi}$ on $\Prym(\widetilde{C}/C)$ defined in Definition \ref{xi} is an incidence polarisation for $A^2(X)$.
\end{proposition}
\begin{proof} 
This follows from \cite[Proposition 3.5]{beau2}. The proof uses intersection theory and does not make use of the characteristic of the base field.
\end{proof}

\begin{corollary}\label{nHIJ}
Let $X$ be a smooth cubic threefold with a good line $l$ over an arbitrary field $k$. Then the intermediate Jacobian of $X$ exists and is given by $(\Prym(X,l), \Xi)$.
\end{corollary}
\begin{proof}
From Proposition \ref{algrep} and Proposition \ref{incpol}, the intermediate Jacobian $J(X_{k^{sep}})$ exists and is given by $(\Prym(X_{k^{sep}},l), \Xi_{k^{sep}})$. By Proposition \ref{GalDescentIJ}, the intermediate Jacobian $J(X)$ also exists and is given by $(\Prym(X,l), \Xi)$.
\end{proof}

\begin{lemma}\label{prereq1}
Let $X$ be a  cubic threefold over a field $k$ of characteristic $0$, and let $\ell$ be a prime. Then Proposition \ref{prereq} holds in this case.
\end{lemma}
\begin{proof}
First of all, $X$ must be non-Hermitian in this case. Thus $\phi$ is an isomorphism by Proposition \ref{algrep}. This gives us a map $$\lambda^2 \circ T_{\ell}(\phi^{-1}): T_\ell \Ab^2_{X_{k^{sep}}} \to \HH^3(X_{k^{sep}}, \mathbb{Z}_{\ell}(2))$$
of $\mathbb{Z}_{\ell}$-modules. By \cite[Proposition 1.3]{colliot}, we know that $2A_0(Y) = 0$ for any cubic threefold $Y$ over any field $F$ containing $k$. Following the proof of \cite[Lemma 1.3]{colliotU}, using the fact that $2A_0(X_{k(X)}) = 0$, we obtain a decomposition of twice the diagonal
$$2\Delta = 2p \times X + Z \in \Ch^3(X \times X)$$
for some $Z$ supported on $X \times D$ where $D \subset X$ is a proper closed subset, and some point $p \in X$. By \cite[Proposition 5.1]{achterBloch}, it follows that $T_{\ell}\lambda^2$ is an isomorphism of $\Gal(\overline{k}/k)$-modules. Thus $\lambda^2 \circ T_{\ell}(\phi^{-1})$ is an isomorphism.
\end{proof}

Next we proceed to construct Prym schemes for families of cubic threefolds.

\begin{proposition}
   Let $\widetilde{C} \xrightarrow{\pi} C \to S$ be a finite Galois morphism of curves over a scheme $S$. This induces a pushforward morphism on relative Jacobians $\Pic^0(\widetilde{C}/S) \xrightarrow{\pi_*} \Pic^0(C/S)$, and this morphism is smooth.
\end{proposition}
\begin{proof}
Consider the pullback morphism $\Pic^0(C/S) \xrightarrow{\pi^*} \Pic^0(\widetilde{C}/S)$ over $S$, which satisfies $\pi_* \circ \pi^* = [2]_{\Pic^0(C/S)}$. We know that $\pi_*$ is a surjective homomorphism of smooth group schemes. Therefore it is flat by miracle flatness. It is also locally of finite presentation, and so it suffices to show for any $p \in \Pic^0(C/S)$ that the fibre $\Pic^0(\widetilde{C}/S)_p \to \Spec(k(p))$ is smooth. 

Let $s: \Pic^0(C/S) \to S$ be the structure morphism. It suffices to show that $(\pi_*)_{s(p)}: \Pic^0(\widetilde{C}/S)_{s(p)} \to \Pic^0(\widetilde{C}/S)_{s(p)}$ is smooth over $\Spec(k(s(p)))$, which follows from Proposition \ref{pushforwardsmooth}.
\end{proof}

\begin{definition}
Let $\widetilde{C} \xrightarrow{\pi} C \to S$ be a finite \'etale morphism of relative curves over $S$. Then $\Ker(\pi_*)^0$ is an abelian scheme, which we call the \emph{Prym scheme} and denote by $\Prym_S(\widetilde{C}/C)$.
\end{definition}

\begin{proposition}\label{polPrym}
If $\widetilde{C} \xrightarrow{\pi} C \to S$ is a finite \'etale morphism of curves over $S$ of degree $2$, then $\Prym_S(\widetilde{C}/C)$ is principally polarised with polarisation $\Xi_S$, satisfying $2\Xi_S = \widetilde{\theta}_S \big |_{\Pic^0(\widetilde{C}/S)}$ when $\widetilde{\theta}_S$ is the canonical polarisation on the relative Jacobian of $\widetilde{C}$.
\end{proposition}
\begin{proof}
We follow the proof of Proposition \ref{algequiv} but over an arbitrary base scheme $S$. Let $i:\Prym_S(\widetilde{C}/C) \to \Pic^0(\widetilde{C}/S)$ be the inclusion. Then we may define the polarisation $\Lambda$ given by the composition

\begin{center}\begin{tikzpicture}[descr/.style={fill=white,inner sep=1.5pt}]
        \matrix (m) [
            matrix of math nodes,
            row sep=3em,
            column sep=2.5em,
            text height=1.5ex, text depth=0.25ex
        ]
{ \Pic^0(C/S) \times \Prym_S(\widetilde{C}/C) & \Pic^0(\widetilde{C}/S) \\
   \widehat{\Pic^0(\widetilde{C}/S)} & \widehat{\Pic^0(C/S)} \times \widehat{\Prym_S(\widetilde{C}/C)}. \\ 
};

        \path[overlay,->, font=\scriptsize,>=latex]
        (m-1-1) edge node[midway, above]{$\pi^* + i$} (m-1-2)
        (m-1-2)  edge[out=355,in=175] node[descr,yshift=0.3ex]{$\widetilde{\theta}_S$} (m-2-1)
        (m-2-1) edge node[midway, above]{$(\widehat{\pi^*},\widehat{i})$} (m-2-2);
\end{tikzpicture}\end{center}

which for any point $p \in S$ specializes to the polarisation $\lambda$ in Proposition \ref{algequiv}. Writing $\Lambda$ as a matrix, we also obtain
$$\Lambda = \begin{pmatrix}
A & B \\
C & D
\end{pmatrix} = \begin{pmatrix}
A & 0 \\
0 & D
\end{pmatrix}$$
because $\Lambda$ is diagonal once we specialize to any point of $S$. We are interested in the polarisation $D: \Prym_S(\widetilde{C}/C) \to \widehat{\Prym_S(\widetilde{C}/C)}$. For any $p \in S$ we have $\Ker(D_{p}) = \Prym_S(\widetilde{C}/C)_p[2]$ so clearly $[2]_{\Prym_S(\widetilde{C}/C)}$ must kill $\Ker(D)$, and so we can factor $D = E \circ [2]_{\Prym_S(\widetilde{C}/C)}$ for some morphism $E: \Prym_S(\widetilde{C}/C) \to \widehat{\Prym_S(\widetilde{C}/C)}$. From Proposition \ref{algequiv}, $E_p$ is a principal polarisation for all $p \in S$ and so $E$ itself is a principal polarisation.
\end{proof}

\begin{definition}
Let $X$ be a family of cubic threefolds over a Noetherian base $S$. Then the \emph{Fano scheme of relative lines} $F(X/S) \to S$ exists. A \emph{family of lines} on $X$ over $S$ is a section $s:S \to F(X/S)$ of this morphism. There is a Zariski open subset $F_{\textrm{good}}(X/S)\subset F(X/S)$ parametrising families of good lines.
\end{definition}

Let $X$ be a family of cubic threefolds over a normal Noetherian base $S$ such that no fibre is isomorphic to a Hermitian cubic threefold. We will associate to $X$ a principally polarised abelian scheme $J(X/S)$ which is the intermediate Jacobian of $X$. 

\begin{proposition}\label{attachprym}
Let $S$ be a normal Noetherian scheme over $\Spec(\mathbb{Z})$. Let $X$ be a smooth cubic threefold over $S$ such that $X_s$ is not Hermitian for any $s \in S$. Then there is a principally polarised abelian scheme $(\Ab^2_{X/S},\Theta)$ over $S$ such that for all points $s \in S$ we obtain $(\Ab^2_{X/S},\Theta)_s \cong (\Ab^2_{X_s},\Theta_s)$ as principally polarised abelian varieties, with $\Theta_s$ being distinguished.
\end{proposition}
\begin{proof}
The proof follows analogously to \cite[Theorem 3.4]{achter}.
\end{proof}

\begin{definition}
Given a family of good lines $s:S \to F_{\textrm{good}}(X/S)$ on a family of cubic threefolds $X$ over a Noetherian base $S$, we can construct the \emph{Prym scheme} $(\Prym_S(X,s),\Xi)$ as in \cite[Theorem 3.4]{achter}, and this is also the intermediate Jacobian of $X$. For Noetherian $T$, given any Cartesian diagram
$$\begin{tikzcd}
X_T \arrow{r} \arrow{d} & X \arrow{d} \\
  T \arrow{r}{f} & S
\end{tikzcd}$$
the morphism $X_T \to T$ is also a family of cubic threefolds, and we have $F(X_T/T) \cong F(X/S)_T$. From $s$ we obtain a natural map $f^*s: T \to F(X_T/T)$ which is a section of the structure map $ F(X_T/T) \to T$ and therefore represents a family of lines on $X_T$. If $s$ is a good family of lines, then so is $f^*s$, since we can check this condition pointwise on the base.
\end{definition}

\begin{proposition}\label{baseChangePrym}
Let $X \to S$ be a family of cubic threefolds over a Noetherian base $S$. For Noetherian $T$, consider a Cartesian diagram 
$$\begin{tikzcd}
X_T \arrow{r} \arrow{d} & X \arrow{d} \\
  T \arrow{r}{f} & S.
\end{tikzcd}$$
Suppose that the Fano scheme of relative lines $F(X/S) \to S$ has a section $s:S \to F(X/S)$ which lands in the locus $F_{\textrm{good}}(X/S)$ of good lines. Then the Prym schemes $(\Prym_S(X,s),\Xi)$ and $(\Prym_T(X_T,f^*s),\Xi_T)$ exist and we have 
$$(\Prym_T(X_T,f^*s),\Xi_T) \cong (\Prym_S(X,s),\Xi)_T.$$
\end{proposition}
\begin{proof}
The existence part is clear. The families of good lines $s,f^*s$ provide us with families of conic bundles $p_S:X' \to \mathbb{P}^2_S$ and $p_T: X_T' \to  \mathbb{P}^2_T$ which fit into a Cartesian diagram 
$$\begin{tikzcd}
X_T' \arrow{r} \arrow{d} & X' \arrow{d} \\
  \mathbb{P}^2_T \arrow{r} & \mathbb{P}^2_S
\end{tikzcd}$$
and so it is clear that the discriminant curves $\Delta(p_S)$ and $\Delta(p_T)$ satisfy $\Delta(p_T) \cong \Delta(p_S)_T$. This yields a commutative diagram of \'etale double covers of families of curves
$$\begin{tikzcd}
\widetilde{\Delta(p_T)} \arrow{r} \arrow{d} & \Delta(p_T) \arrow{d} \\
  \widetilde{\Delta(p_S)} \arrow{r} & \Delta(p_S)
\end{tikzcd}$$
which is Cartesian. It follows that $\Prym_T(X_T,f^*s) \cong \Prym_S(X,s)_T$ as the kernel commutes with base change. As for the polarisations, the polarisation $\Lambda$ constructed in Proposition \ref{polPrym} is stable under base change, as the theta divisor on Jacobians of curves is also stable under base change.
\end{proof}

\begin{remark}
Proposition \ref{baseChangePrym} can be used to remove the normality condition when constructing intermediate Jacobians of families of non-Hermitian cubic threefolds as follows. It suffices to consider trivial families of non-Hermitian cubic threefolds $X \to S$ due to Corollary \ref{del2}. Then there is a morphism $f: S \to U$ so that $X \cong f^*\mathcal{X}$. Then we we observe that $f^*J(\mathcal{X}) \cong J(X)$ by Proposition \ref{baseChangePrym}. The Noetherian requirement could also be weakened, however we do not need these results in this paper.
\end{remark}

\subsection{Intermediate Jacobians of Hermitian cubic threefolds}
In this section we prove that intermediate Jacobians exist for Hermitian cubic threefolds over arbitrary fields. 
Let $X$ be a Hermitian cubic threefold over an algebraically closed field $k$.

\begin{proposition}\label{CH0}
$X$ is universally $\Ch_0$-trivial.
\end{proposition}
\begin{proof}
Since $k$ is algebraically closed, we may assume that $X$ is the Fermat cubic threefold by Proposition \ref{nogoodline}. Then this follows from \cite[Theorem 2.8]{colliot}, whose proof does not rely on the ground field being $\mathbb{C}$.
\end{proof}

\begin{lemma}\label{specialIJT1}
Let $X$ be the Fermat cubic threefold over an algebraically closed field $k$ of characteristic $2$. Then the intermediate Jacobian $(\Ab^2_X,\Theta)$ of $X$ exists. 
\end{lemma}
\begin{proof}
We may lift $X$ to a cubic threefold $\widetilde{X}$ over the witt ring $W(k)$, which is a DVR. Let $\eta$ be the genetic point of $\Spec(W(k))$, and let $s$ be the special point. Consider the intermediate Jacobian of the generic fiber $(\Ab^2_{\widetilde{X}_{\eta}},\Xi_{\eta})$. Since $\widetilde{X}_{\eta}$ is also a Fermat hypersurface, it is universally $\Ch_0$-trivial by \cite[Theorem 2.8]{colliot}. Therefore we can apply \cite[Proposition 5.3]{achter2020decomposition} to spread out to a principally polarised abelian scheme $(\Ab^2_{\widetilde{X}/W(k)},\Xi)$ over $W(k)$ so that
$$(\Ab^2_{\widetilde{X}/W(k)},\Xi)_s \cong (\Ab^2_X,\Xi_s).$$
Since $\Xi_{\eta}$ can be represented by an effective divisor, so can $\Xi_s$, being the specialization of $\Xi$. Moreover, Proposition \ref{CH0} tells us that $X$ is universally $\Ch_0$-trivial, and so from the discussion in \cite[Page 57]{achter2020decomposition}, $\Xi_s$ is distinguished.
\end{proof}

\begin{definition}
Let $f:\mathcal{A}^2_{X/k}\to A$ be a regular homomorphism for a cubic threefold over a field $k$. A cycle $Z \in A^2(X \times A)$ is called \emph{universal} if the induced algebraic map $Z_* : A(k) \to A^2(X)$ is such that $f_k \circ Z_* = 1_A$. 
\end{definition}

\begin{lemma}\label{universalCycle}
Let $X$ be a cubic threefold over a separably closed field $k$. Let $f:\mathcal{A}^2_{X/k}\to A$ be a regular homomorphism with a universal cycle $Z$. If $\dim(A) = \dim(\Ab^2_X)$, then $A$ is an algebraic representative.
\end{lemma}
\begin{proof}
From the universal propery of algebraic representatives we obtain a commutative diagram
$$\begin{tikzcd}
& A^2(X) \arrow{r}{f_k} \arrow{d}{\psi}  & A(k) \\
   A(k)  \arrow{ru}{Z_*} \arrow{r} & \Ab^2_{X}(k) \arrow{ur}{g} &
\end{tikzcd}
$$
where it suffices to show that $g$ is an isomorphism. We obtain $g \circ \psi \circ Z_* = 1_A$ by definition of universal cycles. This means than $\psi \circ Z_*$ is an injection of abelian varieties of the same dimension, hence an isomorphism with inverse $g$.
\end{proof}

\begin{lemma}\label{denseFieldExt}
Let $X$ be a variety which is smooth over a separably closed field $k$. Let $K/k$ be an arbitrary field extension. Then the points $X(k) \subset X_K(K)$ are Zariski dense.
\end{lemma}
\begin{proof}
It follows from the assumptions that $X(k)$ is Zariski dense in $X$. Consider the closure $\overline{X(k)}$ in $X_K$. We claim that $\dim(\overline{X(k)}) = \dim(X) = \dim(X_K)$, from which the lemma follows immediately. This follows from a simple induction on $\dim(X)$.
\end{proof}

\begin{proposition}\label{specialIJT}
Let $X$ be the Fermat cubic over a separably closed field $k$ of characteristic $2$. Then the intermediate Jacobian $(\Ab^2_X,\Theta)$ of $X$ exists. 
\end{proposition}
\begin{proof}
Since $k$ is separably closed, it contains the field $\overline{\mathbb{F}_2}$, which is algebraically closed. The Fermat cubic $X$ is defined over $\overline{\mathbb{F}_2}$. Thus by Lemma \ref{specialIJT1}, the intermediate Jacobian $(\Ab^2_{X/\overline{\mathbb{F}_2}},\Theta)$ exists. By \cite[Proposition 4.3]{achter2020decomposition}, there is a universal cycle $Z \in A^2(\Ab^2_{X/\overline{\mathbb{F}_2}} \times X)$. Let $f: \mathcal{A}^2_{X} \to \Ab^2_X$ be the regular homomorphism. Then by \cite[Proposition 4.3]{achter2020decomposition} we also get that 
$$f_{\overline{\mathbb{F}_2}}: A^2(X) \to \Ab^2_X(k)$$
is an isomorphism, and is given by the pushforward $\widehat{Z}_*$ for some cycle $\widehat{Z} \in A^2((\Ab^2_{X/\overline{\mathbb{F}_2}})^{\vee} \times X)$. Now we base change to $k$ to obtain homomorphisms
$$(\Ab^2_X)_k(k) \xrightarrow{(Z_k)_*} A^2(X_k) \xrightarrow{((\widehat{Z})_k)_*} (\Ab^2_X)_k(k).$$
We claim that $((\widehat{Z})_k)_*$ is regular. Let $T$ be a smooth $k$-scheme and let $g_W: T(k) \to A^2(X_k)$ be an algebraic mapping. We must show that $((\widehat{Z})_k)_* \circ g_W$ arises from a morphism of schemes. We can factor $((\widehat{Z})_k)_*$ as
$$A^2(X_k) \xrightarrow{((\widehat{Z})_k)^*} A^1((\Ab^2_X)_k^{\vee}) \xrightarrow{AJ} (\Ab^2_X)_k(k)$$
where $AJ$ is the Abel-Jacobi map, or more precisely the algebraic representative, noting that $\Pic^0((\Ab^2_X)_k^{\vee}) \cong (\Ab^2_X)_k$. As a result it would suffice to show that $((\widehat{Z})_k)^* \circ g_W$ is algebraic. This follows from a simple intersection theoretic argument which shows that 
$$((\widehat{Z})_k)^* \circ g_W = g_{(\pi_2)_*(\pi_3^*W . \pi_1^*\widehat{Z})}$$
is indeed algebraic, given by the cycle $(\pi_2)_*(\pi_3^*W . \pi_1^*\widehat{Z})$. In particular, $((\widehat{Z})_k)_* \circ (Z_k)_*$ is a homomorphism which agrees with $1_{(\Ab^2_X)_k}$ on the subset $\Ab^2_X(\overline{\mathbb{F}_2}) \subset (\Ab^2_X)_k(k)$, which is dense by Lemma \ref{denseFieldExt}. Therefore this homomorphism is the identity and $Z_k$ is a universal cycle. It follows by Lemma \ref{universalCycle} that $(\Ab^2_X)_k \cong \Ab^2_{X_k/k}$. Now we have to construct the distinguished polarisation. By Lemma \ref{extendPol}, we get that the divisor $\Theta_k$ is distinguished on $\Ab^2_{X_k}$, and so $(\Ab^2_{X_k},\Theta_k)$ is the intermediate Jacobian we seek.
\end{proof}

\begin{proof}[Proof of Proposition \ref{prereq}]
For Cubic threefolds over fields of characteristic $0$ this is Lemma \ref{prereq1}. Otherwise, we may lift $X$ to a cubic threefold $\widetilde{X}$ over the Cohen-Witt ring $W_C(k)$, which has generic point $\eta$ and special point $s$. There is a commutative diagram

 $$\begin{tikzcd}
T_{\ell}A^2(\widetilde{X}_{\overline{\eta}}) \arrow{r}{\lambda^2} \arrow{d} & H^3(\widetilde{X}_{\overline{\eta}},\mathbb{Z}_{\ell}(2)) \arrow{d} \\
 T_{\ell}A^2(X_{k^{sep}}) \arrow{r}{\lambda^2} &  H^3(X_{k^{sep}},\mathbb{Z}_{\ell}(2))
\end{tikzcd}$$
with the right vertical map being an isomorphism, which implies that the bottom horizontal map is surjective. It is also injective by \cite[Theorem 8.1]{achter2022functorial}.

If $X$ is non-Hermitian, then $\phi$ is an isomorphism by Proposition \ref{algrep}. Otherwise, we may assume that $k = k^{sep}$ and that $X$ is the Fermat cubic threefold by Proposition \ref{nogoodline}. We have $\overline{\mathbb{F}_2} \subset k^{sep}$, and by \cite[Proposition 4.3]{achter2020decomposition}, $\phi$ is an isomorphism induced by a cycle $\widehat{Z}$ for the Fermat cubic over $\overline{\mathbb{F}_2}$. From the proof of Proposition \ref{specialIJT} it is evident that $\widehat{Z}_{k^{sep}} = \widehat{Z_{k^{sep}}}$ induces the isomorphism $\phi$ which we seek, in general. This establishes the result in all cases.
\end{proof}

\begin{proposition}\label{HIJ}
Let $X$ be a Hermitian threefold over a field $k$. Then the intermediate Jacobian $(\Ab^2_X,\Theta)$ exists.
\end{proposition}
\begin{proof}
The base change $X_{k^{sep}}$ is isomorphic to the Fermat cubic by Proposition \ref{2bic}. Thus its intermediate Jacobian $J(X_{k^{sep}})$ exists by Proposition \ref{specialIJT}. Then by Proposition \ref{GalDescentIJ}, the intermediate Jacobian of $X$ also exists.
\end{proof}

\subsection{Comparison with the Albanese variety of the Fano surface of lines}
In this section we prove that the second algebraic representative of a cubic threefold is isomorphic to the Albanese variety of its Fano surface of lines. This is a crucial ingredient in proving that intermediate Jacobians of cubic threefolds are stable under specialization. 

Let $X$ be a cubic threefold over a field $k$ which contains a good line $l_0$. Let $p:\widetilde{C} \to C$ be the associated \'etale double cover of the discriminant curve for the conic bundle $\pi: \Bl_{l_0}(X) \to \mathbb{P}^2$. Let $P^+ = \Prym(X,l_0)$ be the associated Prym variety.

\begin{lemma}\label{imm}
Let $X$ be a non-Hermitian cubic threefold over a separably closed field $k$. There exists a closed immersion $\alpha:F(X) \to P^+$ which sends a general line $l \in F(k)$ to $[D(l)-D_0]$. Here $[D(l)] \in  J_5\widetilde{C}$ is the effective divisor describing the $5$ intersection points $ \pi^{-1}(C) \cap l$, and $[D_0] \in  J_5\widetilde{C}$ is fixed so that $\pi_*D_0 \cong H$ and $[D_0]$ belongs to the same connected component as $D(l)$ inside
$$S:= \{x \in  J_5\widetilde{C} : \pi_*x \cong H\}.$$
This map is only well-defined up to translation in the Prym variety.
\end{lemma}
\begin{proof}
When $k$ is algebraically closed, this follows from the proof of  \cite[Proposition 3]{beau} and the Corollary which follows it. The arguments are algebro-geometric and do not make use of the characteristic of the field.

For a separably closed field $k$, consider the morphism $\alpha: F(X_{\overline{k}}) \to (P^+)_{\overline{k}}$. We claim that it is actually defined over $k$, in which case it will follow that the descended morphism $\alpha':F(X) \to P^+$ is also a closed immersion. For any line $l \in F(k)$, the divisor $[D(l)] \in (J_5\widetilde{C})_{\overline{k}}$ consists of $5$ distinct points, which means that the polynomial defining $ \pi^{-1}(C) \cap l$ is separable over $k$. But $k$ is separably closed and so in fact $[D(l)] \in J_5\widetilde{C}$. Moreover, $[D_0]$ can be chosen from $ J_5\widetilde{C}$ as well. This shows that $\alpha$ descends to a morphism $\alpha'$ over $k$ as required.
\end{proof}

We wish to prove an analogous result for Hermitian cubic threefolds.

\begin{definition}
Let $X$ be a Hermitian threefold defined by an equation $F$ over $k$. This is equivalent to a sesquilinear form $\langle,\rangle$ on $V^f\otimes V$, where $V := k^5$ and $V^f$ denotes a twist by frobenius. A \emph{Hermitian point} of $X$ is a point $p$ satisfying
$$\langle p, v \rangle = \langle v, p \rangle^2$$
for all $v \in V$. When $X$ is the Fermat cubic threefold, its Hermitian points are just $X(\mathbb{F}_4)$.

Let $p \in X$ be a Hermitian point. We write $D_p \subset F(X)$ for the divisor of lines passing through $p$, which is indeed a divisor. Let $\ell \subset X$ be a line. We similarly write $D_{\ell} \subset F(X)$ for the divisor of lines incident on $\ell$.
\end{definition}

\begin{lemma}\label{albemb}
Let $X$ be the Fermat cubic threefold over an algebraically closed field $k$ of characteristic $2$. Then the albanese morphism $\alpha: F(X) \to \Alb(F(X))$ is injective on $k$-points.
\end{lemma}
\begin{proof}
Let $\ell$ be a Hermitian line on $X$. That is, $\ell$ is a line spanned by two Hermitian points. Then there are precisely $5$ Hermitian points $x_1 , \dotsc , x_5$ on $\ell$ and for each Hermitian point $x_i$, the incidence variety $D_{x_i} \subset F(X)$ is an elliptic curve by the paragraph following \cite[Corollary 6.6]{cheng}. Let $C = \oplus_iC_i$ be the sum of these elliptic curves. By \cite[Theorem 6.14]{cheng}, there is a chain of purely inseparable isogenies
$$J(C) \xrightarrow{v_*} \Alb(F(X)) \xrightarrow{L_*} \Ab_X^2 \xrightarrow{L^*} \Pic^0_{F(X)} \xrightarrow{v^*} J(C)$$
so it suffices to show that $v^* \circ L^* \circ L_* \circ \alpha$ is injective on $k$-points. This morphism sends a line $\ell'$ to $\mathcal{O}_{F(X)}(D_{\ell'} - D_{\ell}) \big |_{C}$. By \cite[Corollary 6.7]{cheng}, $D_{\ell'}$ intersects each elliptic curve $C_i$ at precisely one point, for $\ell' \in F(X)(k) \backslash D_{\ell}(k)$. Suppose that two lines $\ell_1,\ell_2 \in F(X)(k) \backslash D_{\ell}(k)$ are sent to the same point of $JC$. This happens if an only if the unique line $q_i$ through $x_i$ incident on $\ell_1$ is also incident on $\ell_2$, for all $i = 1 , \dotsc , 5$. This forces $\ell_1,\ell_2,\ell$ to be $3$-coplanar. Clearly, we may select another Hermitian line $\ell_0$ which is not $3$-coplanar with $\ell_1,\ell_2$. Repeating the construction with $\ell_0$, we observe that $v^* \circ L^* \circ L_* \circ \alpha$ sends $\ell_1,\ell_2$ to different points, as required. For general lines $\ell_1,\ell_2 \in F(X)(k)$, we can find a Hermitian line $\ell_0$ such that $\ell_1,\ell_2 \not\in D_{\ell_0}(k)$, and repeat the argument above.
\end{proof}

\begin{proposition}\label{albvsar}
Let $X$ be a cubic threefold over an arbitrary field $k$. The morphism $L_*: \Alb(F(X)) \to \Ab^2_X$ is an isomorphism.
\end{proposition}
\begin{proof}
First of all, the morphism $L_*$ is defined over $k$ since the universal line $L$ is defined over $k$. It suffices to show that the base change $(L_*)_{k^{sep}}$ is an isomorphism.

Suppose that $X$ is non-Hermitian and $k = k^{sep}$, and consider the chain
$$f: \Alb(F(X)) \xrightarrow{L_*} \Ab^2_X \xrightarrow{L^*} \Pic^0(F(X)) \xrightarrow{\_ \cdot \widetilde{C}} J\widetilde{C}$$
where the last arrow is given by intersection with the cycle $[\widetilde{C}] \in \Ch^1(F(X))$. Then $L^* \circ L_*$ sends the class of a line $[\ell]$ to the divisor $D_{\ell}$, and so the map $f$ sends the class of a line $[\ell]$ to $[D(\ell)]$. As a result, $f$ is actually the Abel-Jacobi map induced by the morphism $\alpha: F(X) \to P^+$ described in Lemma \ref{imm}. By \cite[Proposition 9]{beau}, it follows that $f$ is an isomorphism. Thus $L_*$ is a proper monomorphism, and thus a closed immersion by \cite[04XV]{sp}. A closed immersion of abelian varieties of the same dimension is an isomorphism, as required.

Suppose now that $X$ is Hermitian. We may suppose that $k = k^{sep}$ and that $X$ is the Fermat cubic threefold. We may lift it to a cubic threefold $\widetilde{X}$ over the Cohen-Witt ring $W_C(k)$. Let $\eta,s$ be the generic and special points of $W_C(k)$. Consider the algebraic representative $\Ab^2_{\widetilde{X}}$.

We know that $L_* : \Alb(F(\widetilde{X}_{\eta})) \to \Ab^2_{\widetilde{X}_{\eta}}$ is an isomorphism, which spreads out to a morphism $L_*: \Ab^2_{\widetilde{X}} \to \Alb(F(\widetilde{X}))$. This too must be an isomorphism. From the proof of Proposition \ref{specialIJT}, we know that $ \Ab^2_{\widetilde{X}_s} \cong  (\Ab^2_{X_{\overline{\mathbb{F}_2}}})_s$ is the base change of the algebraic representative of the Fermat cubic over $\overline{\mathbb{F}_2}$. It is also implicit in the proof of this proposition that if $\widetilde{X}$ were a lift of $X$ over $\overline{\mathbb{F}_2}$ to $W_C(\overline{\mathbb{F}_2})$, then
$$(\Ab^2_{\widetilde{X}})_{\overline{\mathbb{F}_2}} \cong \Ab^2_{X}.$$ 
Putting this all together, we understand that $(\Ab^2_{\widetilde{X}})_s \cong \Ab^2_{\widetilde{X}_s}$, and so the specialization of $L_*$ is also given by $L_*: \Alb(F(\widetilde{X}_s)) \to \Ab^2_{\widetilde{X}_s}$. This is an isomorphism, as required.
\end{proof}

\begin{proposition}\label{specializeAR}
Let $S$ be a DVR with special point $s$ and generic point $\eta$. Suppose that $X$ is a cubic threefold over $S$. Then $(\Ab^2_X)_s \cong \Ab^2_{X_s}$.
\end{proposition}
\begin{proof}
 Consider the morphism $L_*: \Alb(F(X)) \to \Ab^2_X $, which is an isomorphism due to Proposition \ref{albvsar}. Since Albanese varieties and Fano schemes of lines are stable under base change in this case, we obtain an isomorphism $(L_*)_s: \Alb(F(X_s)) \to (\Ab^2_X)_s$. Again by Proposition \ref{albvsar} we have an isomorphism $L_*: \Alb(F(X_s)) \to \Ab^2_{X_s}$. Thus
$f = L_* \circ (L_*)_s^{-1}: (\Ab^2_X)_s \to  \Ab^2_{X_s}$ is an isomorphism. 
\end{proof}

\section{Applications}
In this section we apply the theory to solve some problems regarding cubic threefolds. We construct a functor from the stack of cubic threefolds to the stack of principally polarised abelian schemes of dimension $5$ which extends the known intermediate Jacobian functor. We also solve a conjecture of Deligne by constructing the intermediate Jacobian of the universal cubic threefold. We prove that smooth cubic threefolds over arbitrary fields are non-rational. We also prove the Torelli theorem for cubic threefolds over arbitrary fields, which states that the intermediate Jacobian of a cubic threefold determines the cubic threefold, up to isomorphism.

\subsection{Preliminaries on automorphisms and cohomology}

Let $\ell$ be a prime different from $\textrm{char}(k)$. We collect some results about the automorphisms and cohomology of cubic threefolds and their intermediate Jacobians.

\begin{remark}\label{defActIJ}
Let $X/S$ and $Y/T$ be families of smooth cubic threefolds and let $S,T$ be smooth and of finite type over an integral Noetherian scheme $\Lambda$. Consider a commutative diagram
$$\begin{tikzcd}
X \arrow{r}{g} \arrow{d} & Y \arrow{d} \\
  S \arrow{r}{f} & T
\end{tikzcd}$$
where $f$ is obtained as an inverse limit of smooth morphisms over $\Lambda$. By Proposition \ref{baseChangeCat}, we know that $(\mathcal{A}^2_{Y/T})_S \cong \mathcal{A}^2_{Y_S/S}$. Then there are natural morphisms $\Ab^2_{X/S} \to  \Ab^2_{Y_S/S} \to \Ab^2_{Y/T}$. Let's denote their composition by $g'$. Then the association $g \mapsto g'$ is functorial. In particular, for any family of cubic threefolds $X/S$, the automorphism group $\Aut(X/S)$ acts naturally on $\Aut(\Ab^2_{X/S}/S)$. 

Let $X$ be a cubic threefold over a field $k$. For all $u,v \in \HH^3(X_{k^{sep}},\mathbb{Z}_{\ell}(2))$ and $\sigma \in \Aut(X_{k^{sep}})$, it is clear that $\sigma(u) \cup \sigma(v) = u \cup v$ since $\sigma$ can only act on $\HH^6_{\textrm{\'et}}(X_{k^{sep}},\mathbb{Z}_{\ell}(4)) \cong \mathbb{Z}_{\ell}(1)$ via the identity. Thus $\Aut(X)$ preserves the distinguished polarisation when it exists. More generally for families of cubic threefolds $X$ over a base $S$ as above, the group $\Aut(X/S)$ preserves the distinguished polarisation when it exists.
\end{remark} 

\begin{proposition}\label{AJcohomology2}
Let $X$ be a cubic threefold over a field $k$, and let $\Ab^2_{X}$ be the second algebraic representative. Let $\ell$ be a prime different from $\textrm{char}(k)$. Then there is an isomorphism
$$\HH^1(\Ab^2_{X_{k^{sep}}},\mathbb{Q}_{\ell}) \cong \HH^3(X_{k^{sep}},\mathbb{Q}_{\ell}(2))^{\vee}$$
which is $\Aut(X_{k^{sep}})$-equivariant and Galois-equivariant.
\end{proposition}
\begin{proof}
This follows from Proposition \ref{prereq}. For the fact that the isomorphism is $\Aut(X_{k^{sep}})$-equivariant, see the proofs of Lemma \ref{cohfunc} and Lemma \ref{bigCommDiagram}, and adapt them to the case of an automorphism $g: X \to X$ and $g': \Ab^2_X \to \Ab^2_X$.
\end{proof}

\begin{lemma}\label{autJ}
Let $X$ be a smooth cubic threefold over a field $k$. Then the map
$\Aut(X_{k^{sep}}) \to \Aut(J(X_{k^{sep}}))$ from Remark \ref{defActIJ} is injective and Galois-equivariant.
\end{lemma}
\begin{proof}
By \cite[Theorem 1.6(2)]{pan2}, $\Aut(X_{k^{sep}})$ acts faithfully on its cohomology group $\HH^3(X_{k^{sep}},\mathbb{Q}_{\ell}(2))$. On the other hand, there is an isomorphism $f: \HH^3(X_{k^{sep}},\mathbb{Q}_{\ell}(2)) \to \HH^1(\Ab^2_{X_{k^{sep}}},\mathbb{Q}_{\ell})^{\vee}$ by Proposition \ref{AJcohomology2}, and this map is equivariant with respect to $\Aut(X_{k^{sep}})$ and $\Gal(k^{sep}/k)$. Therefore $\Aut(X_{k^{sep}})$ must act faithfully on $\Ab^2_{X_{k^{sep}}}$ as well. Therefore the map $\Aut(X_{k^{sep}}) \to \Aut(J(X_{k^{sep}}))$ is injective and Galois-equivariant.
\end{proof}

\begin{proposition}\label{autfermat}
Let $X$ be a Hermitian cubic threefold over a field $k$. Then the automorphism group $\Aut(X_{k^{sep}})$ over the separable closure is isomorphic to the projective unitary group $\textrm{PU}_5(2)$.
\end{proposition}
\begin{proof}
This follows from \cite[5.6]{qbic}.
\end{proof}

\subsection{Construction of the arithmetic Torelli map over $\mathbb{Z}$}
In this subsection we construct the intermediate Jacobian of the universal family of cubic threefolds. This solves a conjecture of Deligne posed in \cite[3.3]{deligne}. We then construct a stack morphism from  smooth cubic threefolds to principally polarised abelian fivefolds which extends the morphism constructed in \cite[Corollary 3.5]{achter}.

\begin{definition}
    The \emph{moduli stack of smooth cubic threefolds} is the quotient $\mathcal{T} = U/\textrm{PGL}_5$, where $U \subset \mathbb{P}^{34}$ is the locus of smooth cubic forms in $5$ variables. This is a Deligne-Mumford stack by \cite[Theorem 1.6]{Benoist2012}. It is smooth over $\Spec(\mathbb{Z})$ due to \cite[Proposition 1.3.1]{benoistThesis}.
\end{definition}

Let $\mathcal{X} \to U$ be the universal cubic threefold. It has the following property. For any trivial family $g:X \to T$ of cubic threefolds over $T$, there is a morphism $f:T \to U$ such that $f^*\mathcal{X} \cong X$. This morphism is unique up to $\textrm{PGL}_5$-action. For the proof of this fact, see \cite[Proposition 28.3.6]{vakil}

 In order to define a morphism $\overline{J}: U \to \mathcal{A}_{5,1}$ of stacks over $\mathbb{Z}$, it suffices to construct the intermediate Jacobian $J(\mathcal{X}/U)$. The reason is as follows. Let $f:T \to U$ be any map. Then $\overline{J}(f)$ is just $J(\mathcal{X}/U) \times_U T$, the pullback across $f$, which is a p.p.a.s over $T$.

\begin{lemma}\label{universalprym}
Consider $W \subset \mathbb{P}^{34}$ the locus of smooth cubic forms in $5$ variables which are not Hermitian. Specifically, $W$ is the complement in $U$ of the vanishing locus of the homogeneous ideal $(2,a_{ijk} | i\neq j, j \neq k, k \neq i)$. Then the intermediate Jacobian $J(\mathcal{X}_W/W)$ exists. 
\end{lemma}
\begin{proof}
Since $W$ is normal Noetherian, we can associate to $\mathcal{X}_W$ its Prym scheme by Proposition \ref{attachprym}, which is also the intermediate Jacobian $J(\mathcal{X}_W/W)$ of $\mathcal{X}_W$.
\end{proof}

\begin{customproof}{of Theorem B}
Let $U \subset \mathbb{P}^{34}_{\mathbb{Z}}$ be the locus of smooth cubic forms in $5$ variables, and let $\mathcal{X} \to U$ be the universal cubic threefold. Let $W \subset U$ be the open subset parametrising non-Hermitian cubic threefolds, which has a complement of codimension at least $2$. Then the algebraic representative $\Ab^2_{\mathcal{X}_W/W}$ exists by Lemma \ref{universalprym}. By Proposition \ref{codim2extend}, we can extend this to an abelian variety $A$ over $U$. Then we have $A \cong \Ab^2_{\mathcal{X}/U}$ from the proof of Proposition \ref{basechange}.

Let $D$ be a relative family of divisors on $\Ab^2_{\mathcal{X}_W/W}$ which induces the canonical principal polarisation from Proposition \ref{attachprym}. Let $L = \mathcal{O}(D)$ be the associated line bundle. Since  $\Ab^2_{\mathcal{X}/U}$ is regular, we can take the closure $\overline{D} \subset \Ab^2_{\mathcal{X}/U}$ which defines a line bundle $L_U = \mathcal{O}(\overline{D})$ extending $L$. This in turn gives us a morphism $\Lambda: \Ab^2_{\mathcal{X}/U} \to (\Ab^2_{\mathcal{X}/U})^{\vee}$ which extends the canonical principal polarisation we started off with. From \cite[Remark 1.10b]{Faltings}, the locus in $U$ where $L_U$ is fiberwise ample is closed and open, and so $L_U$ itself is fiberwise ample. Therefore $L_U$ defines a polarisation of abelian schemes, and the kernel $\Ker_U(\Lambda)$ is a finite flat group scheme. As a result the rank of $\Ker_U(\Lambda)$ is constant on the base, and since it is trivial over an open subset $W$, it must be trivial everywhere so $L_U$ in fact defines a principal polarisation.
\end{customproof}

\begin{proposition}\label{baseChangeClosedPoint}
We have an isomorphism $J(\mathcal{X}/U) \cong (\Ab^2_{\mathcal{X}/U}, L_U)$, with the same notation as in the proof of Theorem B. More generally, for any field $k$ and any morphism $f: \Spec(k) \to U$, we have $f^*(\Ab^2_{\mathcal{X}/U}, L_U) \cong J(f^*\mathcal{X})$.
\end{proposition}
\begin{proof}
We must show, for every point $s \in U$, that $(\Ab^2_{\mathcal{X}/U}, L_U)_s \cong J(\mathcal{X}_s)$. For $s \in W$, this is immediate due to Proposition \ref{attachprym}. Suppose that $s \not \in W$, so that $\mathcal{X}_s$ is Hermitian. Let $s^{sep}$ denote the separable closure of the respective residue field. We first show that $(\Ab^2_{\mathcal{X}/U}, L_U)_{s^{sep}} \cong J(\mathcal{X}_{s^{sep}})$. Since $s^{sep}$ is separably closed, $\mathcal{X}_{s^{sep}}$ is isomorphic to the Fermat cubic threefold $X$. We can lift $X$ to the Fermat cubic threefold $\widetilde{X}$ over the Cohen-Witt ring $W_C(s^{sep})$, which is a DVR. This corresponds to a morphism $f: W_C(s^{sep}) \to U$ since it is a trivial family of cubic threefolds. Now $(f^*(\Ab^2_{\mathcal{X}/U},L_U))_{\eta} \cong J(\mathcal{X}_{\eta})$
if $\eta$ is the generic point of $W_C(s^{sep})$. It follows from the proof of Proposition \ref{specialIJT} that $(\Ab^2_{\mathcal{X}/U}, L_U)_{s^{sep}} \cong J(\mathcal{X}_{s^{sep}})$.

Now let $s$ be a general point and consider $(\Ab^2_{\mathcal{X}/U},L_U)_s$. By Proposition \ref{specializeAR}, We can already identify $J(\mathcal{X})_{s}$ with $(\Ab^2_{\mathcal{X}_s},\Xi)$ for some polarisation $\Xi$, and we have just shown that $\Xi_{s^{sep}}$ is distinguished. Thus $\Xi$ is distinguished, by definition.

The second part of the statement follows from a similar argument.
\end{proof}

\begin{proposition}\label{specialize}
Let $S$ be a DVR with special point $s$ and generic point $\eta$. Suppose that $X$ is a cubic threefold over $S$. Then the intermediate Jacobian $J(X_{\eta}) = (\Ab^2_{X_{\eta}},\Theta_{\eta})$ spreads out to a p.p.a.s. $(A,\Theta)$ over $S$ whose specialization satisfies $(A_s,\Theta_s) = J(X_s)$.
\end{proposition}
\begin{proof}
It is clear that we can take $A = \Ab^2_X$, and $\Theta = \overline{\Theta_{\eta}}$, so it remains to prove the specialization result. We know that $(\Ab^2_X)_s \cong \Ab^2_{X_s}$ by Proposition \ref{specializeAR}.

Since $X$ is a trivial family of cubic threefolds, there is a (non-unique) morphism $f:\Spec(S) \to U$ to the moduli space of cubic forms such that $f^*\mathcal{X} \cong X$. We may consider the principally polarised abelian scheme $f^*J(\mathcal{X})$. From Proposition \ref{baseChangeClosedPoint}, the fibres of $f^*J(\mathcal{X})$ are intermediate Jacobians, and the result follows immediately. 
\end{proof}

\begin{proposition}\label{del}
There is a morphism of stacks $\overline{J}: U \to \mathcal{A}_{5,1}$ which associates to a smooth cubic threefold $X$ over a field $k$ its intermediate Jacobian $J(X/k)$.
\end{proposition}
\begin{proof}
Let $\overline{J}$ be the morphism induced by the principally polarised abelian scheme $J(\mathcal{X}/U)$. Due to Proposition \ref{baseChangeClosedPoint}, this morphism sends a field-valued point $f: \Spec(k) \to U$ to $J(f^*\mathcal{X})$. In particular it sends classes of smooth cubic threefolds over a field to their intermediate Jacobian. 
\end{proof}

\begin{corollary}\label{del2}
There is a morphism of stacks $\widetilde{J}: \mathcal{T} \to \mathcal{A}_{5,1}$ which fits into a commutative diagram
$$\begin{tikzcd}
U \arrow{r} \arrow{dr}{\overline{J}} &\mathcal{T} \arrow{d}{\widetilde{J}} \\
   &   \mathcal{A}_{5,1}
\end{tikzcd}$$
and associates to a smooth cubic threefold $X$ over a field $k$ its intermediate Jacobian $J(X/k)$.
\end{corollary}
\begin{proof}
Since $\mathcal{T} = U/\textrm{PGL}_5$, an $S$-valued object of this category is a $\textrm{PGL}_5$-torsor $\pi:T \to S$ together with an equivariant map $f:T \to U$.

There is some \'etale cover $S_i \to S$ which trivializes this torsor. We obtain trivial torsors $\pi_i: S_i \times \textrm{PGL}_5 \cong T \times_{S} S_i \to S_i$. Consider the intermediate Jacobian $J(\mathcal{X}/U)$ of the universal cubic threefold. We have natural sections $p_i:S_i \to T \times_{S} S_i $ for each $i$ given by the identity element $1 \in \textrm{PGL}_5$. This lets us define the p.p.a.s.'s 
$$A_i := p_i^*f^*J(\mathcal{X}/U).$$ 
These come together with cocyle data 
$$\psi_{ij}:(A_i)_{S_i \cap S_j} \to (A_j)_{S_i\cap S_j}$$
which define a p.p.a.s. $A$ over $S$. We set this as the image of our chosen object under $\widetilde{J}$. This defines a natural morphism of stacks extending the morphism $\overline{J}: U \to \mathcal{A}_{5,1}$.
\end{proof}

\subsection{Non-rationality}
Let $k$ be an algebraically closed field. In this subsection we show that cubic threefolds over $k$ are non-rational. By extension, this shows that any cubic threefold over an arbitrary field is non-rational.

\begin{proposition}\label{upf}
We have the following facts
\begin{enumerate}
    \item Let $A$ be a principally polarised abelian variety. Then $A$ has a unique factorisation into irreducible principally polarised abelian varieties (p.p.a.v), up to isomorphism of p.p.a.v's and re-ordering.
    \item Jacobians of curves are irreducible as p.p.a.v's.
\end{enumerate}
\end{proposition}
\begin{proof}
See \cite[Corollary 2]{debarre} for $(1)$. For $(2)$, if the Jacobian of a curve $C$ is reducible as a p.p.a.v, then the theta divisor on $JC$ is reducible as a variety. But the theta divisor is the image of a symmetric product of $C$ so it is irreducible.
\end{proof}

\begin{proposition}\label{rat3fold}
Let $X$ be a rational threefold over $k$, and suppose that the intermediate Jacobian $J(X)$ exists. Then $J(X)$ is isomorphic, as a principally polarised abelian variety, to a product of Jacobians of curves.
\end{proposition}
\begin{proof}
This is \cite[Proposition 4.6]{beau2}, whose proof does not make use of the characteristic of the ground field.
\end{proof}

\begin{definition}
Let $D$ be an effective divisor on an abelian variety $A$ such that $\lambda_D$ is a principal polarisation. Then we define the \emph{dimension of the singular locus} $\dim(\Sing(\lambda_D))$ as $\dim(\Sing(D))$, which is well-defined by Lemma \ref{dist}.
\end{definition}

\begin{lemma}\label{dimsing}
Let $(A,\lambda_\theta)$ be a principally polarised abelian variety isomorphic to a product of Jacobians of curves with $\theta$ effective. Then the dimension of the singular locus satisfies $\dim(\Sing(\lambda_\theta)) \geq \dim(A)-4$.
\end{lemma}
\begin{proof}
From \cite[Proposition 8]{AM} and \cite[page 212, remark d]{AM}, if $C$ is a connected curve then $\dim(\Sing(JC)) \geq \dim(JC) - 4$. In general, suppose $A \cong JC_1 \oplus JC_2 \oplus B$ where $B$ is an arbitrary p.p.a.v. Then $\theta(JC_1) \boxtimes \theta(JC_2)$ has singular locus of dimension at least $\dim(JC_1 \oplus JC_2) - 2$, given by the intersection of the two theta divisors. Consider the theta divisor $\theta'$ of the principal polarisation on $B$. Then $\Sing(\theta(JC_1) \boxtimes \theta(JC_2)) \times \theta' \subset A$ is singular, of dimension at least $\dim(JC_1 \oplus JC_2 \oplus B) - 3$. As a result $\dim(\Sing(\lambda_\theta)) \geq \dim(A) - 3$.
\end{proof}

\begin{proposition}\label{nonrational}
Let $X$ be a cubic threefold with a good line over an algebraically closed field. Then $X$ is not rational.
\end{proposition}
\begin{proof}
Consider the intermediate Jacobian $J(X) = (\Ab^2_X,\lambda_{\Xi})$ of $X$, which exists by Corollary \ref{nHIJ}. Suppose that $X$ is rational. Then $J(X)$ would be isomorphic to a product of Jacobians of curves by Proposition \ref{rat3fold}. By Lemma \ref{dimsing}, we obtain $\dim(\Sing(\lambda_{\Xi})) \geq 1$. However by Proposition \ref{sing} we have $\dim(\Sing(\lambda_{\Xi})) = 0$ which is a contradiction. 
\end{proof}

\begin{proposition}\label{nonrationalfermat}
The Fermat cubic threefold over a field $k$ is not rational.
\end{proposition}
\begin{proof}
It suffices to consider the case where $k$ is algebraically closed. If $\textrm{char}(k) \neq 2$, then the Fermat cubic has a good line by Proposition \ref{nogoodline}, and so it is not rational by Proposition \ref{nonrational}. We now assume that $\textrm{char}(k) = 2$.

By Proposition \ref{2bic}, the Fermat cubic is isomorphic to the Klein cubic $X = \mathbb{V}_+(x_0x_1^2 + x_1x_2^2 + x_2x_3^2 + x_3x_4^2 + x_4x_0^2)$ over $k$, as the Klein cubic is Hermitian. Then the proof of \cite[Theorem 3]{Beauville} is also valid in characteristic $2$, if we were to use $\ell$-adic cohomology. We provide a rough sketch of this proof. The Klein cubic $X$ has automorphisms
\begin{align*}
& \sigma: [x_0:x_1:x_2:x_3:x_4] \mapsto [x_1:x_2:x_3:x_4:x_0], \\
& \delta: [x_0:x_1:x_2:x_3:x_4] \mapsto [x_0:\zeta x_1:\zeta^{-1} x_2:\zeta^3 x_3:\zeta^6 x_4]
\end{align*}
for some primitive $11^{th}$ root of unity $\zeta$. These satisfy the relations $\sigma^5 = \delta^{11} = \sigma\delta\sigma^{-1}\delta^2 = 1$.
Suppose that $X$ were rational. By Proposition \ref{HIJ} the intermediate Jacobian $JX = (\Ab^2_X,\Theta)$ exists. By Proposition \ref{rat3fold}, it is isomorphic to a product of Jacobians of curves.

Suppose first that $JX \cong (JC, \theta)$ is the Jacobian of a curve. Then $\delta$ and $\sigma$ induce actions on $JX$. Since they are of odd order, they induce further actions $\delta_C,\sigma_C$ on the curve $C$ due to the exact sequence
$$0 \to \mathbb{Z}/2\mathbb{Z} \to \Aut(JC) \to \Aut(C).$$
By Proposition \ref{prereq} we have $\HH^1(\Ab^2_{X},\mathbb{Q}_{\ell}) \cong \HH^3(X,\mathbb{Q}_{\ell}(2))^{\vee}$. By the Lefschetz fixed point theorem, we get $\Tr(\delta | \HH^3(X,\mathbb{Q}_{\ell})) = -1$ and $\Tr(\sigma | \HH^3(X,\mathbb{Q}_{\ell})) = 0$. We deduce that  $\Tr(\delta_C | \HH^1(C,\mathbb{Q}_{\ell})) = -1$ and $\Tr(\sigma_C | \HH^1(C,\mathbb{Q}_{\ell})) = 0$. This means that $\delta_C$ has $3$ fixed points, and $\sigma_C$ has $2$ fixed points. Since $\sigma$ normalizes $\langle \delta \rangle$, the set $C^{\delta_C}$ is permuted by $\sigma_C$. It follows that $C^{\delta_C}$ must be fixed by $\sigma_C$, giving us a contradiction. 

Now suppose that $JX \cong \oplus_i (JC_i ,\theta_i)$ is a direct sum of $2$ or more Jacobians of curves. By Proposition \ref{upf}, $\delta$ acts on the set $\{JC_i\}_i$. This permutation must be trivial since $\dim(JX) = 5 < \ord(\delta) = 11$. Therefore $\delta$ acts on each Jacobian $(JC_i, \theta_i)$, inducing actions on $\HH^1(JC_i,\mathbb{Q}_{\ell})$. However, $\dim(JC_i) < 5$ and so $\dim(\HH^1(JC_i,\mathbb{Q}_{\ell})) < 10$ for each $i$. As a $\mathbb{Q}_{\ell}\delta$-representation, $\HH^1(JC_i,\mathbb{Q}_{\ell})$ must be trivial. This contradicts $\Tr(\delta | \HH^3(X,\mathbb{Q}_{\ell})) = -1$.
\end{proof}

\begin{customproof}{of Theorem A}
We may assume that the ground field $k$ is algebraically closed. If $X$ contains a good line, then $X$ is non-rational by Proposition \ref{nonrational}. Otherwise by Proposition \ref{nogoodline}, $X$ is isomorphic to the Fermat cubic threefold. Then $X$ is non-rational by Proposition \ref{nonrationalfermat}.
\end{customproof}

\subsection{Torelli theorem}

In this section we prove a Torelli theorem for cubic threefolds over arbitrary fields. To be precise, we show that the stack morphism $\widetilde{J}$ constructed in Section $5$ is injective on field-valued points, so that a cubic threefold can be recovered from its intermediate Jacobian.

Let $X$ be a cubic threefold over a field $k$. After choosing a line over $k^{sep}$, we obtain an Albanese embedding $\alpha: F(X)_{k^{sep}} \to \Alb(F(X))_{k^{sep}}$. Consider the morphism $d_{k^{sep}}$ given by the composition
$$(F(X) \times F(X))_{k^{sep}} \xrightarrow{(\alpha,\alpha)} (\Alb(F(X)) \times \Alb(F(X)))_{k^{sep}} \xrightarrow{L_* \circ (\pi_1 - \pi_2)} (\Ab^2_X)_{k^{sep}}$$
over the separable closure $k^{sep}$. This morphism is Galois-equivariant. By Galois descent, we obtain a morphism $d: F(X) \times F(X) \to \Ab^2_X$ over $k$.

\begin{proposition}\label{torellipoint}
Let $X$ be a cubic threefold over a field $k$. The image of $d$ is the theta divisor $\Xi$ on $\Ab^2_X$. Let $T_0(\Xi / \Ab^2_X)$ be the tangent cone of $\Xi$ at $0$. Then $\mathbb{P}(T_0(\Xi / \Ab^2_X))$ is isomorphic to the cubic threefold $X$.
\end{proposition}
\begin{proof}
By Lemma \ref{albemb} and Lemma \ref{imm}, the Albanese embedding $\alpha$ is injective on $k^{sep}$-points. Thus the diagonal $\Delta \subset F \times F$ coincides with $d^{-1}(0)$. By blowing up, we get a diagram
$$\begin{tikzcd}
\mathbb{P}(\mathcal{N}_{\Delta / F \times F}) \arrow{r}{h} \arrow{d} & \mathbb{P}^4 \arrow{d} \\
Bl_{\Delta}(F \times F) \arrow{r} \arrow{d}  & Bl_{0}(\Ab^2_X) \arrow{d}  \\
F \times F \arrow{r}{d}  & \Ab^2_X
\end{tikzcd}
$$
where the top row consists of the exceptional divisors. We claim that the image of $h$ is $X$ up to a linear transformation.

It is known that $\mathbb{P}(\mathcal{N}_{\Delta / F \times F}) \cong \mathbb{P}(T_F)$ where $T_F$ is the tangent bundle of $F$. This is due to the following exact sequence which can be found on \cite[Page 182]{hartshorne}
$$0 \to T_{\Delta} \to T_{F\times F} \otimes \mathcal{O}_{\Delta} \to \mathcal{N}_{\Delta / F \times F} \to 0.$$
 By \cite[Proposition 2.2]{huy}, we have $\mathbb{P}(T_F) \cong L$, the universal line. The map $h$ is uniquely determined by $h^*\mathcal{O}(1)$, which is $\mathcal{O}_{\mathbb{P}(T_F)}(1)$ by basic properties of blowups. The isomorphism $\mathbb{P}(T_F) \cong L$ identifies $\mathcal{O}_{\mathbb{P}(T_F)}(1)$ with $\pi_X^{-1}(\mathcal{O}_X(1))$ which shows that $h$ is the projection map onto $X$. It follows that the image of $h$ is $X$.

Suppose that $X$ in non-Hermitian. Then we can identify $\Ab^2_X \cong P^+$, and from \cite[Proposition 5]{beau} we know that the image of $d_{\overline{k}}$ is the theta divisor $\Xi_{\overline{k}}$. Therefore $\im(d)$ is a divisor on $P^+$ which becomes distinguished upon base change to $\overline{k}$. Thus it follows that $\im(d)$ itself is distinguished, and by their uniqueness we obtain $\im(d) = \Xi$. Thus the image of $h$ is the tangent cone of $\Xi$ at $0$.

Suppose that $X$ is Hermitian. We only need to show that $\im(d) = \Xi$. For this, it suffices to show that $\im(d)_{k^{sep}} = \im(d_{k^{sep}}) = \Xi_{k^{sep}}$. Since $X_{k^{sep}}$ is isomorphic to the Fermat cubic, we may suppose that $k = k^{sep}$ and that $X$ is the Fermat cubic. We may lift $X$ to the Fermat cubic $\widetilde{X}$ over the Cohen-Witt ring $W_C(k^{sep})$. After choosing a line $\ell \subset \widetilde{X}$ flat over $W_C(k^{sep})$, which is possible, we obtain an Albanese embedding $\alpha: F(\widetilde{X}) \to \Alb(F(\widetilde{X}))$. We can use this to construct the map $d:F(\widetilde{X}) \times F(\widetilde{X}) \to \Ab^2_{\widetilde{X}}$. If $\eta$ is the generic point of $W_C(k^{sep})$, then we know that $\im(d_{\eta}) = \Xi_{\eta}$ is the theta divisor on $\Ab^2_{\widetilde{X}_{\eta}}$. We also know that $\im(d_s)$ is the specialization of $\Xi_{\eta}$, and so it is the theta divisor $\Xi_s$ on $\Ab^2_{X}$, and we are done.
\end{proof}

\begin{customproof}{of Theorem C}
This follows immediately from Proposition \ref{torellipoint}.
\end{customproof}

\begin{corollary}\label{Jseparated}
The morphism $\widetilde{J}$ is essentially injective and faithful.
\end{corollary}

\bibliography{Bibliography}{}
\bibliographystyle{amsalpha}

\end{document}